\documentclass[12pt]{amsart}
\usepackage{amsmath, amssymb, amsthm, bbm}
\usepackage{enumitem}
\usepackage{tikz}
\usepackage{graphicx}
\usepackage{xcolor}
\usepackage[UKenglish]{isodate}
\usepackage{graphicx}
\usepackage[font=small,labelfont=it]{caption}
\usepackage{subcaption}
\usepackage{hyperref}
\usepackage{enumitem}
\usepackage{multirow}
\usepackage{fancyvrb}
\usepackage{geometry}
\usepackage{algorithm}
\usepackage{algpseudocode}
\usepackage{algcompatible}
\usepackage{mathtools}

\mathtoolsset{showonlyrefs}

\newcommand{\m}{\mathbbm{m}}
\newcommand{\cL}{\mathcal{L}}

\renewcommand{\phi}{\varphi}
\newcommand{\eps}{\varepsilon}

\newcommand{\E}{\mathsf{E}}
\newcommand{\N}{\mathbb{N}}
\renewcommand{\P}{\mathsf{P}}

\newcommand{\R}{\mathbb{R}}

\newcommand{\ud}{\mathrm{d}}

\def \cC{\mathcal{C}}

\def \cS{\mathcal{S}}
\def \cH{\mathcal{H}}
\def \cF{\mathcal{F}}

\def \cL{\mathcal{L}}

\def \P{\mathsf P}

\def \E{\mathsf E}

\def \N{\mathbb{N}}
\def \R{\mathbb{R}}

\def \ud{\mathrm{d}}
\def \e{\mathrm{e}}

\newtheorem*{theorem*}{Theorem}
\newtheorem{theorem}{Theorem}[section]

\newtheorem{proposition}[theorem]{Proposition}
\newtheorem{corollary}[theorem]{Corollary}
\newtheorem{remark}[theorem]{Remark}
\newtheorem{lemma}[theorem]{Lemma}

\newtheorem{definition}[theorem]{Definition}

\makeatletter
\@namedef{subjclassname@2020}{\textup{2020}Mathematics Subject Classification} 
\makeatother

\title[Integral equations for MFGs]{Integral equations for the optimal boundary surface \\ of a mean-field game of capacity expansion}\thanks{{\bf Acknowledgements}: T.\ De Angelis was partially supported by EU -- Next Generation EU -- PRIN2022 (2022BEMMLZ) CUP: D53D23005780006 and PRIN-PNRR2022 (P20224TM7Z) CUP: D53D23018780001.}

\subjclass[2020]{Primary: 91A16. Secondary: 93E20, 60G40, 45D05, 65R20.}
\keywords{Mean field games; capacity expansion; singular stochastic control; optimal stopping; free-boundary problems; optimal boundary surface; Volterra integral equations; Itô formula; smooth fit; numerical approximation; Picard iteration.}

\author[De Angelis]{Tiziano De Angelis} 
\author[Livieri]{Giulia Livieri}
\author[Ghio]{Maddalena Ghio} 
\address{T.\ De Angelis: School of Management and Economics, Dept.\ ESOMAS, University of Torino, Corso Unione Sovietica, 218 Bis, 10134, Torino, Italy; Collegio Carlo Alberto, Piazza Arbarello 8, 10122, Torino, Italy.}
\email{\href{mailto:tiziano.deangelis@unito.it}{tiziano.deangelis@unito.it}}
\address{G.\ Livieri: The London School of Economics and Political Science, Dept.\ Statistics, Columbia House, 69 Aldwych, London WC2B 4RR.}
\email{\href{mailto:g.livieri@lse.ac.uk}{g.livieri@lse.ac.uk}}
\date{\today}

\begin{document}
\begin{abstract}
We prove that the optimal boundary surface  that splits the action and inaction regions in a mean-field game of capacity expansion studied in (Campi et al.,\ Ann.\ Appl.\ Probab.,\ {\bf 32}(5),\, pp.\,3674-3717, 2022) is the unique continuous solution of a nonlinear integral equation of Volterra type. In order to do that, we first establish continuity of the optimal surface. Then we develop an extension of It\^o's formula which weakens assumptions required in the existing literature on the first-order time-derivative and/or second-order space derivative of the value function. The paper also provides an algorithm for the numerical solution of the integral equation and we compute optimal controls numerically for the mean-field game. 
\end{abstract}

\maketitle

\newpage

\section{Introduction}

We obtain a nonlinear integral equation of Volterra type for the optimal boundary surface that splits the action and inaction regions in a mean-field game (MFG) of capacity expansion with singular controls (SCs). The MFG was studied by Campi et al.\ in \cite{CDAGLAAP2021}, who showed the existence of the optimal surface but left open questions about its continuity and possible methods for its numerical computation. Our work is mostly theoretical and involves the study of fine properties of the optimal surface and a generalisation of It\^o's formula which we believe to be new. However, we also show that the integral equations we obtain are amenable to numerical solution. Therefore, we are able to compute numerically the optimal control in the MFG, which is not a common occurrence in the literature, especially when dealing with SCs.

The derivation of integral equations is a familiar procedure for optimal stopping (OS) boundaries; cf.\ Peskir and Shiryaev \cite{PS2006} for many solved examples. This is also used in SC, although to a lesser extent, via a well known connection between OS and SC; cf.\ \cite{DeAFF} for an early example in capacity expansion problems. The overarching idea is to obtain a probabilistic expression for the value function $V$ of the stochastic control/stopping problem in which the optimal boundary features explicitly.  Then, in OS, evaluating $V$ at the optimal boundary one obtains an integral equation for the boundary itself; for SC problems instead $V$ must be replaced by a suitable derivative of $V$ (further details in Section \ref{sec::TheIntegralEquation}). The first step requires an application of It\^o's formula, which is generally not straightforward because of a lack of regularity of both $V$ and the optimal boundary.

The approach described in the paragraph above, to our knowledge, has not been explored in the MFG literature and this paper is a first attempt in that direction. In particular, the main technical difficulties we face in transferring ideas from the classical optimal stopping literature into our framework are: (i) lack of sufficient regularity for the value function of the MFG (denoted $v(t,x,y)$ in what follows); (ii) lack of knowledge concerning fine properties of the optimal surface (denoted $b(t,y)$ hereafter). 

Concerning (i), the regularity of the value function can only be addressed in part. From the connection between SC and OS we know that $u\coloneqq\partial_y v$ is the value function of an auxiliary OS problem (which still features the mean-field interaction term from the SC problem). The variable $y$ enters $u(t,x,y)$ as a parameter, in the sense that there is no dynamics associated to that variable. Then, an application of classical It\^o's formula would require continuity of derivatives $\partial_t u$, $\partial_x u$, $\partial^2_{xx} u$. Generalisations of It\^o's formula as in, e.g., \cite{peskir2005change} would require continuity of $\partial_t u$, $\partial_x u$, $\partial^2_{xx} u$ in the inaction/continuation region, plus some control on $\partial^2_{xx} u$ along the boundary $b(t,y)$ (convexity of $u$ in $x$ would suffice). Crucially, this is missing in our context.

As for (ii), one contribution of this paper is to prove that the mapping $(t,y)\mapsto b(t,y)$ is continuous (Theorem \ref{prop:jointlycontinuityb}), which was not known before. The continuity of optimal boundaries as functions of more than one variable is a rather delicate issue. In some cases, PDE theory can be used but not in our case, because the underlying dynamics in the MFG is completely degenerate in the $y$-direction (i.e., no diffusive component). Outside the PDE realm only a few general results are known in this direction (cf.\ Cai et al. \cite{CDeAP}). Continuity is crucial for establishing that the optimal surface $b(t,y)$ is indeed the {\em unique} solution to the integral equation (Theorem \ref{th:UniquenessTheoremIntegralEquation}), which in turn yields some theoretical guarantees for the numerical solution of that equation. 

Going back to (i), we are able to prove that $\partial_x u$ is a continuous function (Proposition \ref{prop:continuityspatialderivative}) but we lack valuable insight into the regularity of $\partial_t u$ and $\partial^2_{xx}u$---the two things are connected, of course. In particular, we only know that $\partial_t u$ and $\partial^2_{xx}u$ exist in the inaction/continuation region as weak derivatives in $L^p$, but we do not control their bounds/integrability across the optimal boundary $b(t,y)$. Thus, we devise a new extension of It\^o's formula
that refines and improves the one presented in Cai and De Angelis \cite{CDAarXiv2022}, where stronger assumptions of continuity of $\partial_t u$ and $\partial_x u$ are required but in a $d$-dimensional setting with any $d\ge 1$.

In summary, our main contributions are: (a) we prove that the optimal  boundary surface in the MFG from \cite{CDAGLAAP2021} is the unique continuous solution of a nonlinear integral equation of Volterra type; (b) we devise a new extension of It\^o's formula which is tailored for applications in OS and SC problems, relaxing assumptions on time-derivative and second-order space-derivative of the value function which are generally difficult to verify in practice; (c) we solve numerically the integral equations and obtain the optimal surface and the optimal control for the MFG.

\subsection{Contribution to the literature}

In this subsection, we discuss our contributions to the literature on generalisations of It\^o's formula and numerical methods for MFGs.

Extensions of It\^o's formula have been the object of intense study for many years. A full review of the subject falls outside of the scope of this introduction and we will instead focus on some papers that propose change-of-variable formulae for specific applications in OS or SC (an interested reader may refer to the Introduction of \cite{CDAarXiv2022} for a broader overview). It is often the case that for applications in stochastic control and OS one wishes to derive change-of-variable formulae that resemble Tanaka's formula as closely as possible. One reason is that the so-called smooth-fit condition (i.e., the continuity of spatial derivatives across the boundary) is used to select the optimal boundary; as a result the local-time terms in Tanaka's formula vanish for the correct ``candidate'' optimal boundary. A widely used formula in this spirit is due to Peskir \cite{peskir2005change} (see also \cite{Pe07} for multi-dimensional versions). As mentioned above, Peskir's formula requires continuity of derivatives in the continuation/inaction set plus some knowledge about spatial convexity/concavity of the value function (or a control on the second order spatial derivative at the boundary). Elworthy et al.\ \cite{ETZ} obtain change-of-variable formulae of a similar kind but under the assumption that left-derivatives in time and space of the value function be of bounded variation. Popular in stochastic control is also It\^o-Krylov formula (\cite[Ch.\ 2, Sec.\ 10]{Kr}), which is tailored for solutions of Hamilton-Jacobi-Bellman equations. It requires that the value function belongs to a Sobolev space with time/space gradient and spatial Hessian matrix whose all entries are functions in $L^p$ for suitable $p>0$. None of the key assumptions in the papers mentioned here are satisfied in our setting. Thus, we devise a change-of-variable formula that complements existing ones and may be useful in other problems of stochastic control. We also refer the reader to Remark \ref{rem:cvf} below for a detailed technical comparison between our result and the one obtained in \cite{CDAarXiv2022}. 

Numerical methods play a key role in a wide range of applications of MFGs, especially because most models do not have explicit or semi-explicit solutions. A review of several aspects concerning numerical methods for MFGs and mean field controls with regular controls\footnote{By regular controls we mean those having a bounded impact on the velocity of the underlying dynamics.} is provided by Lauriere \cite{LAMS2021}. Numerical methods for MFGs with singular controls are less studied and present additional difficulties due to unbounded changes in the velocity of the underlying process and possible discontinuities in the state trajectories. In particular, to our knowledge, existing works have mainly focused on proving existence of solutions for general classes of MFGs involving singular controls (e.g., \cite{DFFNMOR2023, FSICON2017, FHSICON2017}), on establishing a relation between MFGs of singular controls and the associated $N$-player games (e.g., \cite{CGLNACO2023,CGSPA2022}) and on developing a probabilistic framework for extended MFGs of singular controls (e.g., \cite{DH2025SICON,FSICON2017}). Nonetheless, since the seminal paper by Cardaliaguet and Hadikhanloo \cite{CH2017ESAIM}, which focuses on regular controls, only a limited number of papers have tackled the convergence of learning procedures for MFGs involving singular controls---specifically, iterative schemes that provide a constructive solution method. These include the study of Dianetti et al. \cite{DFFNMOR2023} based on submodularity of the game's structure and the use of Tarski's fixed point theorem, the work by Campi et al. \cite{CDAGLAAP2021} which leverages the connection between OS and SC and is similar in spirit to the construction used in Gu{\'e}ant \cite{GMMMAS2012} in an analytical study of MFG equations
with quadratic Hamiltonian, and the very recent work of Dianetti et al. \cite{dianetti2025entropy} where the authors establish the convergence of the fictitious play scheme for MFGs arising from SC problems associated with stochastic optimal control models. In particular, the numerical experiments carried out in Section \ref{sec:numericalexperiments} of this paper build on the learning procedure proposed in Campi et al. \cite{CDAGLAAP2021}.

\subsection{Structure of the paper} The paper is organised as follows. In Section \ref{sec:model} we recall the setting and main results of the MFG studied by Campi et al.\ \cite{CDAGLAAP2021}. We collect there only the essential facts that are necessary for our own analysis. In Section \ref{sec:properties_vf_and_boundaries} we prove new results concerning the value function and the optimal boundary of the MFG. In particular, we prove continuity of the function $\partial_x u$ in Proposition \ref{prop:continuityspatialderivative} and continuity of the boundary $b(t,y)$ in Theorem \ref{prop:jointlycontinuityb}. In Section \ref{sec::TheIntegralEquation} we obtain the main result of the paper, i.e., the characterisation of $b(t,y)$ as the unique solution of an integral equation (Theorem \ref{prop:Integralequation}). This is where we develop our main technical contribution, with a suitable extension of It\^o's formula. Finally, in Section \ref{sec:numericalexperiments} we present in detail a numerical algorithm for the solution of the integral equation. We also present numerous plots of the optimal surface, of the optimal control and of the overall efficiency of the algorithm.

\section{Model description and summary of known results}\label{sec:model}
In this section, we review the model of \textit{mean-field games of finite fuel capacity expansion with singular controls} proposed in \cite{CDAGLAAP2021}. 
Let $\Pi=(\Omega, \mathcal{F}, \mathbb{F}=(\mathcal{F}_t)_{t\geq 0}, \bar{\P})$ be a filtered probability space satisfying the usual conditions and supporting a one-dimensional $\mathbb{F}$-Brownian motion $W$; the initial $\sigma$-field $\mathcal{F}_0$ is not necessarily trivial. Let $\Sigma:=\R\times[0,1]$ and $\Sigma^\circ:=\R\times(0,1)$. We denote by $\mathcal P(\Sigma)$ the class of probability measures on $\Sigma$ and by $\mathcal P_2(\Sigma)$ the subclass of probability measures with finite second moment. Let $(X_0, Y_{0-})$ be a two-dimensional $\mathcal{F}_0$-measurable random variable with joint law $\nu\in\mathcal{P}(\Sigma)$ independent of the Brownian motion, and let $\xi\in\Xi^{\Pi}(Y_{0-})$ be an admissible strategy, where for a $\mathcal{F}_0$-measurable random variable $Z \in [0,1]$ the set $\Xi^\Pi(Z)$ is defined as 
\begin{equation*}
\begin{aligned}
    \Xi^\Pi(Z):= \big \{ \xi:\:&\text{$(\xi_t)_{t\ge 0}$ is $\mathbb F$-adapted, non-decreasing,  right-continuous,}\nonumber\\
    &\text{ with $\xi_{0-}=0$ and $Z+\xi_t\in [0,1]$ for all $t\in[0,T]$, $\bar{\P}$-a.s.} \big \}.
    \end{aligned}
\end{equation*}
Then, given a bounded Borel measurable function $m : [0,T]\mapsto [0,1]$, for $t \in [0,T]$ we define the two-dimensional, degenerate, controlled dynamics:
\begin{equation}\label{eq:Xdynamics}
\begin{aligned}
    X_t         &= X_0 + \int_{0}^{t}a(X_s, m(s))\,\ud s + \int_{0}^{t}\sigma(X_s)\,\ud W_s,\\
    Y_t^{\xi}   &= Y_{0-}+\xi_t,
\end{aligned}
\end{equation}
where functions $a(\cdot)$ and $\sigma(\cdot)$ will be specified later. 

The goal of the ``representative player" consists of maximizing over the set of all admissible strategies $\xi \in \Xi^{\Pi}(Y_{0-})$ the following objective functional
\begin{equation}\label{eq:objfunctional}
    \mathcal{J}(\xi) = \bar{\E}\left[\int_{0}^{T}\e^{-rt}f(X_t, Y_t^{\xi})\,\ud t-\int_{[0,T]}\e^{-r t} c_0\,\ud\xi_t\right],
\end{equation}
where $\bar{\E}[\cdot]$ is the expectation under $\bar{\P}$, $f$ is a running payoff function which we specify later, $c_0>0$ is a cost, and $r \geq 0$ a discount rate. 
The integral with respect to the positive random measure $\ud \xi$ includes possible atoms at the initial and terminal time (corresponding to jumps of $\xi$). The value of the optimization problem for the representative player is denoted by
\begin{equation*}
    V^{\nu} = \sup_{\xi \in \Xi(Y_{0-})} J(\xi).
\end{equation*}
Now we recall the definition of solution for the MFG of capacity expansion, provided in \cite{CDAGLAAP2021}.
\begin{definition}[Solution of the MFG of capacity expansion]\label{def:MFGsol} 
A solution of the MFG of capacity expansion with initial condition $(X_0,Y_{0-})\sim\nu\in\mathcal P_2(\Sigma)$ is a pair $(m^* , \xi^*)$ with $m^* :[0,T]\rightarrow[0,1]$ a measurable function and $\xi^* \in \Xi(Y_{0-})$ such that:
\begin{enumerate}
\item[(i)] {\em(Optimality property)} $\xi^*$ is optimal, i.e.,
\begin{equation}
J(\xi^* )=V^{\nu}=\sup_{\xi \in\Xi}\bar{\mathsf{E}}\left[\int_0^{T}\e^{-rt}f(X^*_t,Y^\xi_t)\,\ud t-\int_{[0,T]} \e^{-rt}c_0\,\ud \xi_t\right],
\nonumber
\end{equation}
where $(X^*,Y^\xi)$ is a solution of  \eqref{eq:Xdynamics} associated to $(m^*,\xi)$.
\item[(ii)] {\em(Mean-field property)} Letting $(X^* , Y^* )$ be the solution of \eqref{eq:Xdynamics} associated to $(m^*, \xi^*)$, the consistency condition holds:
\[ 
m^* (t)=\bar{\mathsf{E}}[Y^*_t],\quad\text{for each $t\in[0,T]$.}
\] 
\end{enumerate}
We say that a solution $\xi^*$ of the MFG is in \emph{feedback form} if $\xi^*_t=\eta(t,X^*,Y_{0-})$, $t\in[0,T]$, for some non-anticipative measurable mapping 
\[
\eta:[0,T]\times C([0,T]; \mathbb R)\times[0,1] \to [0,1]
\]
(i.e., such that $\eta(t,X^*,Y_{0-})=\eta(t,X^*_{\cdot\wedge t},Y_{0-})$).
\label{def:solMFG}
\end{definition}
For the solution of the MFG it is convenient to embed the problem in a Markovian framework. Let us set
\begin{equation*}
    \P_{0, x, y}(\,\cdot\,):=\bar{\P}(\,\cdot\,|X_0=x, Y_{0-}=y),
\end{equation*}
and assume that the mapping $(x,y)\mapsto \P_{0, x, y}(A)$ is measurable for any $A\in\cF$.
Since $(X_0, Y_{0-})\overset{d}{\sim} \nu$ then
\begin{equation*}
    \bar{\P}(\,\cdot\,)=\int_{\Sigma}\P_{0, x, y}(\,\cdot\,)\nu(\ud x, \ud y)\quad\quad\text{and}\quad\quad\bar{\E}[\,\cdot\,]=\int_{\Sigma} \mathsf{E}_{0, x, y}[\,\cdot\,]\nu(\ud x, \ud y).
\end{equation*}
The dynamics in \eqref{eq:Xdynamics} conditional upon the initial data $(X_t,Y^\xi_{t-})=(x, y)\in \Sigma$ at time $t\in[0,T]$ reads
\begin{equation}\label{eq:XYdynamicsnew}
\begin{aligned}
    X_{t+s} &= x + \int_{0}^{s}a(X_{t+u}, m(t+u))\,\ud u + \int_{0}^{s}\sigma(X_{t+u})\,\ud W_{t+u}\\
    Y_{t+s}^{\xi}   &= y+(\xi_{t+s}-\xi_{t-}),\quad s\in[0,T-t],
\end{aligned}
\end{equation}
where we notice that $\ud W_{t+u}=\ud (W_{t+u}-W_{t})$.
To keep track of the initial data, we use the notation $(X^{t,x}_{t+s},Y^{t,x,y;\xi}_{t+s})_{s\in[0,T-t]}$ where the $Y$-dynamics may depend on $x$ via the control $\xi$. 
For the process started at time zero (i.e., $t=0$), we use the simpler notation $(X_s^x, Y_s^{x,y;\xi})_{s \in [0,T]}$.

Consistently  with the notation introduced so far, we use $\P_{t, x, y}(\cdot) = \bar{\P}(\,\cdot\, | X_t=x, Y_{t-}=y)$ for simplicity,
and $\P_{x,y}=\P_{0,x,y}$ for the special case $t=0$. When no confusion arises we drop the subscript from $\P_{t, x, y}$ and $\E_{t, x, y}$ and simply use $\P$ and $\E$. 

For each initial condition $(t,x,y)\in  [0,T]\times\Sigma$ we introduce the optimization problem. 
\begin{equation}\label{eq:P0}
\begin{split}
v(t,x,y) &:= \sup_{\xi \in \Xi_{t,x}(y)}\,J(t,x,y ; \xi)\quad\text{with}\\
J(t,x, y; \xi) &:= \E_{t,x,y}\left[\int_{0}^{T-t}\!\! \e^{- r s} f(X_{t+s}, Y_{t+s}^{\xi})\,\ud s -\, \int_{[0,T-t]}  \e^{-r s} c_0 \ud\xi_{t+s}\right],
\end{split}
\end{equation}
where the set $\Xi_{t,x}(y)$ of admissible controls is defined by 
\begin{equation}\label{eq:Xitxy}
\begin{aligned}
&\Xi_{t,x}(y):= \big \{ \xi:\: \text{$\xi_{u}=0$ for $u \in [0,t)$,}\\
&\qquad\qquad\qquad\:\text{$(\xi_{t+s})_{s\ge 0}$ is $(\mathcal F_{t+s})_{s\ge 0}$-adapted, non-decreasing, right-continuous,}\\
&\qquad\qquad\qquad\:\:\text{with $y+\xi_{t+s}\in [0,1]$ for all $s\in[0,T-t]$, $\P_{t,x,y}$-a.s.} \big \}.
\end{aligned}
\end{equation}

The link between the problem above and the original MFG was established in \cite{CDAGLAAP2021} and it reads as follows:
\[
V^\nu=\int_\Sigma v(0,x,y)\nu(\ud x,\ud y).
\]
Moreover, letting
$\mu_t^{x,y;\xi}:=\mathcal{L}(X_t^{x}, Y_t^{x, y;\xi})\in\mathcal{P}(\Sigma)$ be the law of the pair $(X_t^{x}, Y_{t}^{x, y; \xi})$ for an arbitrary control $\xi$, the consistency condition between $m(t)$ and $Y^\xi$ can be expressed as
\begin{equation*}
    m(t) = \int_{\Sigma}\mathsf{E}_{x,y}[Y_t^{\xi}]\nu(\,\ud x,\,\ud y)=\int_{\Sigma}\int_{\Sigma}y'\mu_t^{x,y;\xi}(\,\ud x',\,\ud y')\nu(\,\ud x,\,\ud y),\quad t\in[0,T].
\end{equation*}

\subsection{Assumptions and solution of the MFG of capacity expansion.}\label{subsec::assumptions}
We now introduce our assumptions. We first give assumptions on the coefficients of the SDE (see \ref{itm:A1}) and on the profit function (see \ref{itm:A2}). Then, we impose some integrability assumptions (see \ref{itm:A3}).
\begin{enumerate}[label=(A\arabic*)]
\item\label{itm:A1} ({\bf Dynamics}) The functions $a\,:\,\Sigma\rightarrow\mathbb{R}$ and $\sigma\,:\,\mathbb{R}\rightarrow(0,\infty)$ satisfy the following conditions:
\begin{itemize}
    \item[(\,i\,)] There is $L(a)>0$ such that
    \[
    \big|a(x,m)-a(x',m')\big|\le L(a)\big(|x-x'|+|m-m'|\big), 
    \]
    for $x,x'\in\R$ and $m,m'\in[0,1]$. Moreover,
$x\mapsto a(x,m)\in C^{1}(\mathbb{R})$ for each $m\in[0,1]$.
\item[(\,ii\,)] $\sigma(\cdot)\in C^1(\R)$ with $|\partial_{x}\sigma(x)|\le L(\sigma)$ for some $L(\sigma)>0$.
    \item[(\,iii\,)] The mapping $m \mapsto a(x, m)$ is non-decreasing on $[0, 1]$ for all $x \in \mathbb{R}$.
\end{itemize}
\end{enumerate}

Assumptions \ref{itm:A1}-(i) and \ref{itm:A1}-(ii) guarantee existence and uniqueness of a strong solution to \eqref{eq:XYdynamicsnew} which induces a $\P$-a.s.\ continuous flow 
$(t, x, s) \mapsto X^{t,x}_{t+s}$ (e.g., \cite{KS1988}, pp.~397--398, or \cite{B2017}, Theorem 9.9). Moreover, the spatial regularity of $a(\cdot, m)$ and $\sigma(\cdot)$ ensures that the stochastic flow $x \mapsto X^{t,x}$ is continuously differentiable with $Z^{t,x}:=\partial_x X^{t,x}$ given by (see \cite{PP1990}, Chapter V.7)
\begin{equation*}
\begin{aligned}
Z^{t,x}_{t+s}&=\exp\left(\int_{0}^{s}\Big(\partial_x a(X_{t+u}^{t, x}, m(t+u))-\tfrac12[\partial_x\sigma(X^{t,x}_{t+u})]^2\Big)\,\ud u+\int_0^s\partial_x\sigma(X^{t,x}_{t+u})\ud W_{t+u} \right)\\
&=:\exp\left(\int_{0}^{s}\partial_x a(X_{t+u}^{t, x}, m(t+u))\,\ud u\right)\mathcal{M}^{t,x}_{s},
\end{aligned}
\end{equation*}
where $\mathcal{M}^{tx}_{s}$ is a true martingale thanks to boundedness of $\partial_x \sigma$. 
It is clear that $(t, x,s) \mapsto Z^{t,x}_s$ is a continuous flow by continuity of the flow $(t, x,s) \mapsto X^{t,x}_s$. 
Finally, \ref{itm:A1}-(iii) is a technical assumption required by \cite{CDAGLAAP2021}
in the construction of the optimal control in the MFG. 
\begin{enumerate}[label=(A\arabic*), start=2]
   \item\label{itm:A2} ({\bf Running cost}) We have $f \in C^{2}(\Sigma^{\circ})$ and either $\sigma(x) = \sigma$ is constant or $x \mapsto \partial_{x y}f(x, y)$ non-increasing. Furthermore, we assume
\begin{itemize}
    \item[(\,i\,)] \textit{Monotonicity:} $x\mapsto f(x,y)$, $y \mapsto f(x,y)$ are increasing and $\partial_{x y} f > 0$ on $\Sigma^\circ$ with
    \begin{equation*}
        \lim_{x\rightarrow -\infty}\partial_{y} f(x,y)<r c_0< \lim_{x\rightarrow +\infty}\partial_{y} f(x,y),\quad\forall y\in[0,1].
    \end{equation*}
    \item[(\,ii\,)] \textit{Concavity:} $y\mapsto f(x,y)$ is strictly concave $\forall x \in \mathbb{R}$.
    \end{itemize}
\smallskip

\item\label{itm:A3} ({\bf Integrability}) The following integrability conditions hold:
\begin{itemize}
    \item[(\,i\,)] There exists $p>1$ such that, given any Borel measurable $m : [0,T]\rightarrow [0,1]$ and letting $X$ be the associated solution of the SDE \eqref{eq:Xdynamics} we have
\begin{equation*}
    \mathsf{E}_{t, x, y}\left[\int_{0}^{T-t}\e^{-r s}\left(|f(X_{t+s}, y)|^{p}+|\partial_y f(X_{t+s}, y)|^{p}\right)\,\ud s\right]<\infty,
\end{equation*}
for all $(t, x, y)\in [0,T]\times \mathbb{R} \times [0,1]$.
    \item[(\,ii\,)] For any compact $K \subset [0, T] \times \Sigma^{\circ}$,
    \begin{equation*}
        \sup_{(t, x, y) \in K} \mathsf{E}_{t, x}\left[\int_{0}^{T-t} \e^{-r s} \left(|\partial_{yy} f(X_{t+s}, y)|+(1+Z_{t+s})|\partial_{x y} f(X_{t+s},y)|\right)\ud s\right] < \infty.
    \end{equation*}
\end{itemize}
\end{enumerate}
The set of assumptions in \ref{itm:A2} is in line with the literature on irreversible investment and it is fulfilled, for example, by profit functions of Cobb-Douglas type (i.e., $f(x,y)=x^{\alpha}y^{\beta}$ with $\alpha \in [0,1]$, $\beta \in (0,1)$ and $x > 0$). The integrability condition in \ref{itm:A3}-(i) guarantees that the problem is well posed and will allow us the use of the dominated convergence theorem in some of the technical steps of the proofs. 
\vspace{0.1cm}

\subsection{Solution of the MFG of capacity expansion}\label{sec:solution}
 
In this section we review the main result from \cite{CDAGLAAP2021} (cf.~\cite{CDAGLAAP2021}, Theorem 2.5) and some of the steps in its proof, which we are going to use in the present paper. Our Assumption \ref{itm:A1} is more restrictive than what is required by \cite{CDAGLAAP2021}, but for simplicity we state the theorem under such stronger assumption.

\begin{theorem}[Solution of the MFG of capacity expansion]\label{teo:existenceSolMFG} Suppose Assumptions \ref{itm:A1}, \ref{itm:A2}, and \ref{itm:A3}-(i) hold. 
Then, there exists an upper-semicontinuous function $c: [0,T] \times \R \to [0,1]$, with $t\mapsto c(t,x)$ and $x\mapsto c(t,x)$ both non-decreasing, such that the pair $(m^*, \xi^*)$ with
\[ \xi_t ^* := \sup_{0\le s\le t} (c(s,X^* _s)-Y_{0-})^+, \quad m^* (t) := \int_\Sigma \E_{x,y} \left[Y_t ^* \right] \nu(\ud x,\,\ud y), \quad t \in [0,T],\]
is a solution of the MFG as in Definition \ref{def:solMFG}.
\end{theorem}

For the proof of the theorem, \cite{CDAGLAAP2021} propose an iterative scheme which is shown to converge to the MFG solution.  
The scheme is initialized by setting $m^{[-1]}(t) \equiv 1$ for $t \in [0,T]$. At the $n$-th step, given a non-decreasing and right-continuous function $m^{[n-1]}:[0,T]\rightarrow [0,1]$ the following dynamics is introduced
\begin{equation}\label{eq:dynamicsn}
\begin{aligned}
   X_{t+s}^{[n]; t, x} &= x + \int_{0}^{s}a(X_{t+u}^{[n]; t, x}, m^{[n-1]}(t+u))\,\ud u + \int_{0}^{s}\sigma(X_{t+u}^{[n]; t, x})\,\ud W_{t+u} \\
   Y_{t+s}^{[n]; t, x,y} &= y + \xi_{t+s},
\end{aligned}
\end{equation}
for $(x,y)\in\Sigma$, $s \in [0,T-t]$, $t \in [0,T]$ and $\xi \in \Xi_{t, x}(y)$. 
Then, the authors of \cite{CDAGLAAP2021} study the following singular control problem (indicated by $\textbf{SC}^{[n]}_{t,x,y}$): 
\begin{equation}
\begin{aligned}
v_n(t,x,y):= & \sup_{\xi\in\Xi_{t,x}(y)}J_n(t,x,y;\xi)\qquad\text{with}\label{SCn-1}\\
J_n(t,x,y;\xi):= &\,\E_{t,x,y}\Big[\int_0^{T-t}\e^{-rs}f(X^{[n]}_{t+s},y + \xi_{t+s})\,\ud s-\int_{[0,T-t]}\e^{-rs}c_0\,\ud \xi_{t+s}\Big].
\end{aligned}
\end{equation}
It is shown that $\partial_y v_n(t,x,y)=u_n(t,x,y)$ with $u_n$ the value function of the following optimal stopping problem (indicated by $\textbf{OS}^{[n]}_{t,x,y}$):
\begin{equation}\label{eq:un}
\begin{aligned}
u_n(t,x,y):= & \inf_{\tau\in\mathcal{T}_t}U_n(t,x,y;\tau)\qquad\text{with}\\
U_n(t,x,y;\tau):= &\,\E_{t,x}\left[\int_0^{\tau}\e^{-rs}\partial_yf(X_{t+s}^{[n]},y)\,\ud s+c_0\e^{-r\tau}\right],\quad\text{for $\tau\in\mathcal{T}_t$},
\end{aligned}
\end{equation}
where $\mathcal{T}_t$ is the set of stopping times for the filtration generated by $(W_{t+s}-W_t)_{s\geq 0}$, 
with values in $[0, T-t]$. 
The associated continuation region, $\mathcal{C}^{[n]}$, and stopping region, $\mathcal{S}^{[n]}$, read
\begin{equation}
\begin{aligned}
\mathcal{C}^{[n]}&:= \lbrace (t,x,y) \in [0,T] \times \Sigma\,:\,u_n(t,x,y)<c_0 \rbrace,\\
\mathcal{S}^{[n]}&:= \lbrace (t,x,y) \in [0,T] \times \Sigma\,:\,u_n(t,x,y)=c_0 \rbrace.
\end{aligned}
\end{equation}
Then, the minimal optimal stopping time for ${\bf OS}^{[n]}_{t,x,y}$ is found to be
\begin{equation}\label{eq:stoppingtime}
\begin{aligned}
    \tau_{*}^{[n]}(t, x, y)&=\inf\{s \in [0,T-t]\,:\,u_n(t+s, X_{t+s}^{[n]},y)= c_0\}\\
    &=\inf\{s \in [0,T-t]\,:\,c_n(t+s, X_{t+s}^{[n]})\ge y\},
\end{aligned}
\end{equation}
for a function $(t,x)\mapsto c_n(t,x)$ which is upper semi-continous with $t\mapsto c_n(t,x)$ and $x\mapsto c_n(t,x)$ right-continuous and non-decreasing. It is then shown that the optimal control for $\textbf{SC}^{[n]}_{t,x,y}$ reads
\[
\xi_{t+s} ^{[n]*} := \sup_{0\le u\le s} (c(t+u,X^{[n]}_{t+u})-Y_{0-})^+.
\]
For $t=0$, setting $Y^{[n]*}_s:=y+\xi^{[n]*}_s$ one defines
\[
m^{[n]}(s):=\int_\Sigma\E_{x,y}\left[Y^{[n]*}_s\right]\nu(\,\ud x,\,\ud y).
\] 
The map $s\mapsto m^{[n]}(s)$ is non-decreasing and right-continuous (by dominated convergence) with values in $[0,1]$, so it can be used to define processes $(X^{[n+1]}, Y^{[n+1]})$ and the value function $v_{n+1}$ of problem $\textbf{SC}^{[n+1]}_{t,x,y}$ by iterating the above construction.

It is shown in \cite{CDAGLAAP2021} that the sequence $(u_n)_{n\ge 0}$ is decreasing and it converges to a {\em continuous} function $u$. Likewise, they show that the sequence $(c_n)_{n\ge 0}$ is decreasing and it converges to a function $c$. Using these facts they then prove that the iterative scheme converges to the solution of the MFG, in the sense that $(X^{[n]},Y^{[n]*},m^{[n]})$ converges pointwise to $(X^*,Y^*,m^*)$ from Definition \ref{def:solMFG}. In the limiting procedure they also prove that $(Y^*,m^*)$ are expressed as in Theorem \ref{teo:existenceSolMFG}. An ancillary result concerns the pair $(u,c)$: $u(t,x,y)$ is the value function of an optimal stopping problem of the form as in \eqref{eq:un} but with $X^{[n]}$ therein replaced by $X^*$; the function $c(t,x)$ is the optimal stopping boundary for such problem.

The main object of interest in this paper is the boundary $c(t,x)$ that determines the optimal actions of the representative agent in the MFG. Such boundary arises from the study of the optimal stopping problems ${\bf OS}^{[n]}_{t,x,y}$ and therefore it is convenient to recall a few facts obtained in \cite{CDAGLAAP2021}, concerning the value functions $u_n$. Again, the results in \cite{CDAGLAAP2021} are proven under slightly more general assumptions than \ref{itm:A1}--\ref{itm:A2} but for simplicity we restrict to our setup.
\begin{proposition}\label{prop:properties_vf_1}
    Let Assumptions \ref{itm:A1}, 
        \ref{itm:A2}, and \ref{itm:A3}-(i) hold. Then the value function of the optimal stopping problem {\em $\textbf{OS}_{t,x,y}^{[n]}$} has the following properties:
    \begin{itemize}
        \item[(i)] $0\leq u_n(t, x, y) \leq c_0$ for $(t,x,y)\in[0,T]\times\Sigma$; 
        \item[(ii)] The map $x\mapsto u_n(t, x, y)$ is non-decreasing for each fixed $(t,y)\in [0,T]\times[0,1]$ and $y\mapsto u_n(t, x, y)$ is non-increasing for each $(t,x)\in[0,T]\times\mathbb{R}$;
        \item[(iii)] The map $t\mapsto u_n(t, x, y)$ is non-decreasing for each fixed $(x,y)\in\Sigma$;
        \item[(iv)] We have $u_n \in C([0,T]\times\Sigma')$ with $\Sigma'=\R\times(0,1]$. 
    \end{itemize}
\end{proposition}
\begin{proof}
    See \cite{CDAGLAAP2021}, Proposition 3.2.
\end{proof}

Given a set $U$, we denote $W^{1,2;p}(U)$, $1\le p<\infty$, the class of functions that are in $L^p(U)$ along with their weak derivatives of first order in time and space and of second order in space; we denote $C^{0,1;\alpha}(U)$, $\alpha\in(0,1)$, the class of $\alpha$-H\"older continuous functions with $\alpha$-H\"older-continuous space derivative; finally, we use $W^{1,2;p}_{\ell oc}(U)$ and $C^{0,1;\alpha}_{\ell oc}(U)$ for functions belonging to $W^{1,2;p}(K)$ and $C^{0,1;\alpha}(K)$, respectively, for any compact set $K\subset U$.

Let $\mathcal{C}_{y}^{[n]}:=\{(t, x):(t, x, y) \in \mathcal{C}^{[n]}\}$. In the proof of \cite[Lemma A.1]{CDAGLAAP2021}, it is shown that $u_n(\,\cdot\,,y) \in W_{\ell oc}^{1, 2; p}(\mathcal{C}_{y}^{[n]})$, for any $p\in[1,\infty)$, and it solves 
\begin{equation}\label{eq:weak_PDE}
   \left(\partial_t u_n + \frac{\sigma^2(\,\cdot\,)}{2} \partial_{x x} u_n + a(\,\cdot\,,m^{[n-1]}(\,\cdot\,))\partial_{x} u_n - r u_n\right)(\,\cdot\,,y) = -\partial_{y} f(\,\cdot\,,y),\quad\text{a.e.\ in $\cC^{[n]}_y$}.
\end{equation}
By Sobolev embedding $W_{\ell oc}^{1, 2;p}(\mathcal{C}_y^{[n]}) \hookrightarrow \mathcal{C}_{\ell oc}^{0,1;\alpha}(\mathcal{C}_y^{[n]})$ for $\alpha = 1 - 3/p$ with $p > 3$. Then, we deduce  $u(\,\cdot\,,y) \in \mathcal{C}_{\ell oc}^{0,1;\alpha}(\mathcal{C}_y^{[n]})$. 

Finally, it is worth recalling the relationship between $\cC^{[n]}$ and the boundary $c_n$. As it transpires from the above discussion, we have
\[
\cC^{[n]}=\{(t,x,y): c_n(t,x)<y\}\quad\text{and}\quad\cS^{[n]}=\{(t,x,y): c_n(t,x)\ge y\}.
\]
Analogously, we have
\begin{equation*}
\begin{aligned}
\cC&\coloneqq\{(t,x,y): u(t,x,y)<c_0\}=\{(t,x):c(t,x)<y\},\\
\cS&\coloneqq\{(t,x,y): u(t,x,y)=c_0\}=\{(t,x):c(t,x)\ge y\}.
\end{aligned}
\end{equation*}
The function $u$ enjoys the same regularity as functions $u_n$ and it solves the same PDE as \eqref{eq:weak_PDE} but with $\cC_y$ in place of $\cC^{[n]}_y$ and $m^*$ in place of $m^{[n-1]}$ (cf.\ \cite[Lemma A.1]{CDAGLAAP2021}).

\section{Refined properties of value function and optimal boundaries in $\textbf{OS}_{t,x,y}^{[n]}$}\label{sec:properties_vf_and_boundaries}
In this section we prove new results concerning functions $u_n$ and $c_n$ which are required for our numerical scheme. 
We prove continuity of $\partial_x u_n$ on $[0,T] \times \Sigma^{\circ}$ (Proposition \ref{prop:continuityspatialderivative}) and continuity of the generalized left-continuous inverse of $c_n(t,\cdot)$ (Theorem \ref{prop:jointlycontinuityb}).

The proof of Proposition \ref{prop:continuityspatialderivative} below requires the following result\footnote{In writing the present paper, we realized that the statement of Lemma 4.4 in \cite{CDAGLAAP2021} is not accurate. In the present work, for the sake of completeness, we report the precise statement. In particular, the null set in \eqref{eq:convergencetau} depends on the sequence $(t_k, x_k, y_k)_{k \in \mathbb{N}}$.} from \cite[Lemma 4.4]{CDAGLAAP2021}.
\begin{lemma}\label{lemm:continuityoftau}
    Fix $(t, x, y) \in [0,T] \times \Sigma^{\circ}$ and let $(t_k, x_k, y_k) \rightarrow (t, x, y)$ as $k \rightarrow \infty$. Then, 
\begin{equation}\label{eq:convergencetau}
         \lim_{k \rightarrow \infty} \tau_{*}^{[n]}(t_k, x_k, y_k) = \tau_{*}^{[n]}(t, x, y),\quad \P-\emph{a.s.}
    \end{equation}
     with $\tau_{*}^{[n]}(t, x, y) = 0$, $\P$-\emph{a.s.} for $(t, x, y) \in \partial \mathcal{C}^{[n]}$.
\end{lemma}

\begin{proposition}\label{prop:continuityspatialderivative}
    Let Assumptions \ref{itm:A1}, 
        \ref{itm:A2}, and \ref{itm:A3} hold. Then, $\partial_{x} u_n\in C([0, T] \times \Sigma^{\circ})$. 
\end{proposition}
\begin{proof}
Fix an arbitrary $(t, x, y) \in [0,T] \times \Sigma^{\circ}$ and let $\tau_{*}^{[n]} = \tau_{*}^{[n]}(t, x, y)$. For any small $\varepsilon>0$ we have:
\begin{equation*}
    \begin{split}
        &u_n(t, x+\varepsilon, y)-u_n(t, x, y)\\
        &\leq \E\left[\int_{0}^{\tau_{*}^{[n]}}\e^{-r s}\left(\partial_y f(X_{t+s}^{[n]; t, x+\varepsilon},y)-\partial_y f(X_{t+s}^{[n]; t, x}, y)\right)\,\ud s\right]\\
        &=\int_{x}^{x+\varepsilon}\E\left[\int_{0}^{\tau_{*}^{[n]}}\e^{-rs}\partial_{x y} f(X_{t+s}^{[n];t,\eta}, y)Z_{t+s}^{[n];t,\eta}\,\ud s\right]\,\ud \eta,
    \end{split}
\end{equation*}
where the equality holds by the fundamental theorem of calculus and Fubini's theorem.
Dividing by $\varepsilon$ and letting $\varepsilon \rightarrow 0$, we use dominated convergence (cf.\ Assumptions \ref{itm:A3}-(i),(ii)) and continuity of the flows $x \mapsto (X^{[n]; t, x}, Z^{[n];t,x})$ to conclude
\begin{equation*}
    \limsup_{\varepsilon \rightarrow 0} \frac{u_n(t, x+\varepsilon, y)-u_n(t, x, y)}{\varepsilon} \leq \E\left[\int_{0}^{\tau_{*}^{[n]}}\e^{-r s}\partial_{x y}f(X_{t+s}^{[n];t, x}, y)Z_{t+s}^{[n];t,x}\,\ud s\right].
\end{equation*}

Letting $\tau_{*}^{[n] \varepsilon}:=\tau_{*}^{[n]}(t, x+\varepsilon, y)$ and arguing similarly to the above paragraph we have
\begin{equation*}
    \begin{split}
        &u_n(t, x+\varepsilon, y)-u_n(t, x, y)\\
        &\geq \E\left[\int_{0}^{\tau_{*}^{[n]\varepsilon}}\e^{-r s}\left(\partial_y f(X_{t+s}^{[n]; t, x+\varepsilon},y)-\partial_y f(X_{t+s}^{[n]; t, x},y)\right)\,\ud s\right]\\
&=\int_{x}^{x+\varepsilon}\E\left[\int_{0}^{\tau_{*}^{[n]\varepsilon}}\e^{-r s}\partial_{x y}f(X_{t+s}^{[n];t, \eta},y)Z_{t+s}^{[n];t,\eta}\,\ud s\right]\,\ud \eta.
    \end{split}
\end{equation*}
Dividing again by $\varepsilon>0$ and letting $\varepsilon \rightarrow 0$, we can now invoke Lemma \ref{lemm:continuityoftau} to justify that $\tau_{*}^{[n];\varepsilon}\rightarrow\tau_{*}^{[n]}$ and obtain 
\begin{equation*}
    \liminf_{\varepsilon \rightarrow 0} \frac{u_n(t, x+\varepsilon, y)-u_n(t, x, y)}{\varepsilon} \geq \E\left[\int_{0}^{\tau_{*}^{[n]}}\e^{-r s}\partial_{x y}f(X_{t+s}^{[n];t, x},y)Z_{t+s}^{[n];t,x}\,\ud s\right].
\end{equation*}

So, in conclusion we have shown that $\partial_{x} u_n$ exists in $[0,T] \times \Sigma^{\circ}$ and it reads
\begin{equation*}
    \partial_x u_n(t, x, y) = \E\left[\int_{0}^{\tau_{*}^{[n]}}\e^{-r s}\partial_{x y}f(X_{t+s}^{[n]; t, x}, y)Z_{t+s}^{[n];t,x}\,\ud s\right].
\end{equation*}
Continuity of $(t, x, y) \mapsto \tau_{*}^{[n]}(t, x, y)$ (in the sense of Lemma \ref{lemm:continuityoftau}) and of the stochastic flow $(t,s,x) \mapsto (X^{[n];t, x}_s, Z^{[n];t,x}_s)$, combined with dominated convergence implies that $\partial_x u_n$ is continuous on $[0,T] \times \Sigma^{\circ}$.
\end{proof}

It may be tempting to prove regularity of $\partial_{t} u_n$ in a similar fashion as we did for $\partial_x u_n$. 
This would imply taking limits of $(X_{t+s}^{[n]; t, x}-X_{t+s-\varepsilon}^{[n]; t-\varepsilon, x})/\eps$ as $\eps\to 0$, which would require additional regularity conditions on $m^{[n-1]}(t)$. Since the latter depends on the optimal boundary $(t, y)\mapsto c_n(t, y)$, it is unclear that such higher regularity can actually be proven. We will come back to the implications of this issue in the next section.

Now we focus on the study of regularity of the optimal stopping boundaries.
Let 
\begin{equation}\label{eq:setH}
    \mathcal{H}:=\{(x, y)\in \mathbb{R} \times [0,1]\,:\,\partial_{y}f(x,y)-r c_0 < 0\}.
\end{equation}
It is well known in optimal stopping theory that $[0,T)\times\cH\subset\cC^{[n]}$ because for any $(x,y)\in\cH$ and any stopping time $0<\tau\le \inf\{s\ge 0:(X^{[n]}_{t+s},y)\notin\cH\}\wedge(T-t)$,
\[
u_{n}(t,x,y)\le c_0+\E\Big[\int_0^\tau\e^{-rs}\Big(\partial_y f(X^{[n]}_{t+s},y)-rc_0\Big)\ud s\Big]<c_0.
\]

Recall that $\partial_{xy}f>0$. For fixed $y \in [0, 1)$ we define 
\begin{equation}
\label{eq::xbarquantity}
\bar{x}(y):=\inf\{x \in \mathbb{R}\,:\,\partial_{y} f(x, y) - r c_0 > 0\}=\sup\{x\in\mathbb R:(x,y)\in\mathcal H\}.
\end{equation}
We prefer to work with generalized inverses of the optimal boundaries, because for those functions we can establish continuity. 
We define:
\begin{equation}\label{eq:generalized}
    b(t,y):=\sup\{x\in\mathbb{R}\,:\,c(t,x)<y\}\quad\text{and}\quad b_n(t,y):=\sup\{x\in\mathbb{R}\,:\,c_n(t,x)<y\},
\end{equation}
with $\sup\varnothing=-\infty$.
Since $x\mapsto c_n(t,x)$, $x\mapsto c(t,x)$, $t\mapsto c_n(t,x)$ and $t\mapsto c(t,x)$ are non-decreasing and right-continuous, then $t\mapsto b_n(t,y)$ and $t\mapsto b(t,y)$ are non-increasing for $y\in[0,1]$ while $y\mapsto b_n(t,y)$ and $y\mapsto b(t,y)$ are non-decreasing for $t\in[0,T]$. It is proven in Step 2 of the proof of Lemma A.1 in \cite{CDAGLAAP2021} that $y\mapsto b(t,y)$ is also continuous\footnote{Regrettably, the proof in \cite{CDAGLAAP2021} contains some typos which have been fixed in \cite{campi2020mean} (see Step 2 in the proof of Lemma 5.1)} on $(0,1]$. Analogous arguments allow us to conclude that $y\mapsto b_n(t,y)$ is continuous on $(0,1]$. Finally, we can express the sets $\cS$ and $\cS^{[n]}$ in terms of $b$ and $b_n$ as follows:
\begin{equation}\label{eq:SSn}
\cS=\{(t,x,y)\in[0,T]\times \Sigma: x \ge b(t,y) \}\quad \text{and}\quad \cS^{[n]}=\{(t,x,y)\in[0,T]\times \Sigma: x \ge b_n(t,y) \}.
\end{equation}

Our first result concerns continuity of the boundary $b$.

\begin{theorem}\label{prop:jointlycontinuityb}
Let Assumptions \ref{itm:A1}, 
\ref{itm:A2}, and \ref{itm:A3}-(i) hold. Then, for all $n\in\mathbb N$ the mappings $(t, y) \mapsto b_n(t,y)$ and $(t, y) \mapsto b(t,y)$ are jointly continuous on $[0,T]\times(0,1]$ with $b_n(T,y)=b(T,y)=\bar x(y)$ for all $y\in(0,1]$.
\end{theorem}
\begin{proof}
    It suffices to prove the claim for $b$ as the proof is analogous for $b_n$. Since $y\mapsto b(t,y)$ is continuous, if we prove that also $t\mapsto b(t,y)$ is continuous on $[0,T]$ we then obtain joint continuity of $(t,y)\mapsto b(t,y)$ on $[0,T]\times(0,1]$ thanks to monotonicity of the maps $t\mapsto b(t,y)$ and $y\mapsto b(t,y)$ (cf.\ \cite[Prop.\ 1]{KDAMM1969}).

First, we prove that $(t, y) \mapsto b(t, y)$ is lower semi-continuous. It will then follow that $t \mapsto b(t, y)$ is right-continuous for every $y \in [0,1]$, because it is non-increasing. Fix $(t, y)\in[0,T]\times[0,1]$ and take a sequence $(t_k, y_k)_{k \in \mathbb{N}}$ that converges to $(t, y)$ as $k \rightarrow \infty$. We denote $x_k:=b(t_k, y_k)$. Then $(t_k, x_k, y_k) \in \mathcal{S}$ for every $k \in \mathbb{N}$. Since $\mathcal S$ is closed, in the limit we get $\liminf_{k \rightarrow \infty} (t_k, x_k, y_k) =  (t, \liminf_{k \rightarrow \infty} b(t_k, y_k), y) \in \mathcal{S}$. Therefore, $\liminf_{k \rightarrow \infty} b(t_k, y_k) \geq b(t, y)$ by \eqref{eq:SSn}. 

Next, we prove that $t \mapsto b(t, y)$ is left-continuous. We proceed by contradiction as in \cite{DASIAM2015}. Assume that there exists $(t_0, y_0) \in (0, T] \times (0,1)$ such that $b(t_0, y_0) < b(t_{0}-, y_0)$. Let $x_{1}^{0}$ and $x_{2}^{0}$ be two points such that $b(t_0, y_0) < x_{1}^{0} < x_{2}^{0} < b(t_{0}-, y_0)$ and let $R:= [0, t_0) \times (x_{1}^{0}, x_{2}^{0})$. Then, $R \times \{y_0\} \subset \mathcal{C}$ because $t\mapsto b(t,y_0)$ is non-increasing. In particular, \eqref{eq:weak_PDE} holds a.e.\ in $R$. Let $\varphi \in \mathcal{C}_{c}^{\infty}((x_1^{0}, x_2^{0}))$, $\varphi \geq 0$ with $\int\varphi(x)dx=1$ and multiply \eqref{eq:weak_PDE} by $\varphi$. Since $\partial_t u(\,\cdot\,, y) \geq 0$ a.e.\ on $R$ for every $y \in [0,1]$ (Proposition \ref{prop:properties_vf_1}-(iii)) we have a.e.\ on $R$: 
\begin{equation*}
\varphi(\cdot)\left[\frac{\sigma^2(\cdot)}{2}\partial_{xx}u(\cdot,y_0)  + a(\cdot, m^*(\cdot))\partial_{x}u(\cdot, y_0)-r u(\cdot, y_0)\right]\leq -\varphi(\cdot)\partial_{y} f(\cdot, y_0).
\end{equation*}

Then, for arbitrary $0\le t< t_0$ and $0< h<t_0-t$, using integration by parts we obtain:
\begin{equation}
\begin{aligned}
&\frac{1}{h}\int_{t}^{t+h}\int_{x^0_1}^{x^0_2} \left( \frac{\sigma^2(x)}{2}\varphi (x)\partial_{xx}  + a(x, m^*(s))\varphi(x)\partial_x-r\varphi(x) \right) u(s,x,y_0) \ud x\, \ud s\\
&=\frac{1}{h}\int_{t}^{t+h}\int_{x^0_1}^{x^0_2} \left( \tfrac{1}{2}\partial_{xx}\big[\sigma^2(x)\varphi (x)\big]  + a(x, m^*(s))\varphi(x)\partial_x-r\varphi(x) \right)u(s, x, y_0) \ud x\, \ud s\\
& \leq - \frac{1}{h}\int_{t}^{t+h}\int_{x^0_1}^{x^0_2} \partial_y f(x, y_0)\varphi(x) \ud x\, \ud s. 
\end{aligned}
\end{equation}
Thanks to continuity of $u(\cdot,y_0)$, $\partial_x u(\cdot,y_0)$ and $a(\cdot)$, and right-continuity of $m^*(\cdot)$, letting $h\rightarrow 0$ we obtain
\begin{equation*}
\begin{aligned}
&\int_{x^0_1}^{x^0_2} \left( \tfrac{1}{2}\partial_{xx}\big[\sigma^2(x)\varphi (x)\big]  + a(x, m^*(t))\varphi(x)\partial_x-r\varphi(x) \right)u(t, x, y_0) \ud x \\
& \leq - \int_{x^0_1}^{x^0_2} \partial_y f(x, y_0)\varphi(x)\ud x 
\end{aligned}
\end{equation*}
Letting $t\uparrow t_0$ and using that $u$ and $\partial_x u$ are continuous (cf.\ Proposition \ref{prop:continuityspatialderivative})
with $u(t_0,x,y_0)=c_0$ and $\partial_x u(t_0,x,y_0)=0$, and that the limit $m^*(t_0-)=\lim_{t\uparrow t_0}m^*(t)$ is well-defined, we obtain
\begin{equation*}
\int_{x^0_1}^{x^0_2} \left( \tfrac{1}{2}\partial_{xx}\big[\sigma^2(x)\varphi (x)\big]-r\varphi(x) \right)c_0 \ud x  \leq - \int_{x^0_1}^{x^0_2} \partial_y f(x, y_0)\varphi(x)\ud x. 
\end{equation*}
Integrating by parts yields
\begin{equation}\label{final}
\int_{x^0_1}^{x^0_2}(\partial_yf(x,y_0)-rc_0)\varphi(x) \ud x\leq 0.
\end{equation}
Hence $\partial_yf(x,y_0)-rc_0\leq 0$ for all $x\in(x^0_1,x^0_2)$ by arbitrariness of $\varphi\geq 0$ and continuity of $x\mapsto\partial_yf(x,y_0)-rc_0$. However, $\mathcal{S}\subseteq \{(x,y): \partial_yf(x,y)-rc_0\ge 0\}$ (cf.\ the paragraph after \eqref{eq:setH}). Then it must be $\partial_yf(x,y_0)-rc_0= 0$ for all $x\in(x^0_1,x^0_2)$, which contradicts $\partial_{xy}f>0$. This proves that $t\mapsto b(t,y_0)$ is continuous on $[0,T]$. 

In particular, the above result yields continuity of $b(\cdot,y_0)$ at $t_0=T$. Since $b(t,y_0)\ge \bar x(y_0)$ for all $t\in[0,T)$ due to $\mathcal H\subseteq\mathcal C$, then $b(T,y_0)\ge \bar x(y_0)$. If, however $b(T,y_0)>\bar x(y_0)$, then it must be $\partial_yf(x,y_0)-rc_0>0$ for $x\in(\bar x(y_0),b(T,y_0))$ which contradicts \eqref{final}. Thus we also conclude that $b(T,y_0)=\bar x(y_0)$.
\end{proof}

We recall from \cite{CDAGLAAP2021} that the sequence $(c_n)_{n\in\mathbb N}$ is non-increasing and it converges pointwise to $c$.
As a consequence, for all $(t,y)$ the sets $\{x\in\R\,:\,c_n(t,x)<y\}$ are increasing in $n \in \mathbb{N}$ and the corresponding generalised inverses $b_n(t,y)=\sup\{x\in\R\,:\,c_n(t,x)<y\}$ form a non-decreasing sequence. The latter converges (pointwise) to a limit $b_\infty\le b$, where the inequality holds because $c_n \geq c$ implies $b_n\leq b$. Arguing by contradiction let us assume that for some $(t,y)$ we have $b_\infty(t,y)<b(t,y)$ and let us take $b_\infty(t,y)<x<b(t,y)$. Then, it should be $c(t,x)<y<c_n(t,x)$ for all $n\in\N$ but that is impossible because $c_n\downarrow c$. Then, $b_n\uparrow b$ pointwise as $n\to\infty$. 
Since $b_n$ and $b$ are continuous and the convergence is monotone,
by Dini's theorem we have the next simple corollary.
\begin{corollary}\label{cor:unifconv}
The sequence $(b_n)_{n\in\mathbb N}$ is non-decreasing and it converges uniformly to $b$ on any compact subset of $[0,T]\times(0,1]$.
\end{corollary}

\section{The integral equation}\label{sec::TheIntegralEquation}

In this section we derive integral equations for the optimal stopping boundary $b_n$ of problem $\textbf{OS}_{t,x,y}^{[n]}$ (Theorem \ref{prop:Integralequation}) and subsequently for the optimal boundary $b$ arising from the solution to the MFG (Corollary \ref{cor:inteq}). 
Then, Theorem \ref{th:UniquenessTheoremIntegralEquation} proves uniqueness of the solution to such integral equations in suitable classes of functions.

\begin{theorem}[The Integral Equation]\label{prop:Integralequation}
Suppose Assumptions \ref{itm:A1}--\ref{itm:A3} hold. Then, the value function $u_n(t, x, y)$ of the optimal stopping problem $\mathbf{OS}_{t,x,y}^{[n]}$ in \eqref{eq:un} has the following representation:
    \begin{equation}\label{eq:integral_rep}
    \begin{aligned}
    u_n(t,x,y)&=\e^{-r(T-t)}c_0 + \mathsf{E}_{t,x}\Big[\int_{0}^{T-t}\e^{-r s}\partial_{y}f(X_{t+s}^{[n]},y)\,\ud s \Big]\\
            &\quad+ \E_{t,x}\Big[\int_{0}^{T-t}\e^{-rs}\left(r c_0 - \partial_y f(X_{t+s}^{[n]},y)\right)\mathsf{1}_{\{X_{t+s}^{[n]} \geq b_n(t+s,y)\}}\,\ud s\Big],
\end{aligned}
\end{equation}
for all $(t, x, y) \in [0,T]\times\Sigma$. 

The optimal stopping boundary $b_n$ solves the following integral equation:
\begin{equation}\label{eq:integral_eq0}
    \begin{aligned}
    c_0\big(1-\e^{-r(T-t)}\big)&= \mathsf{E}_{t,b_n(t,y)}\Big[\int_{0}^{T-t}\e^{-r s}\partial_{y}f(X_{t+s}^{[n]},y)\,\ud s \Big]\\
            &\quad+ \mathsf{E}_{t,b_n(t,y)}\Big[\int_{0}^{T-t}\e^{-rs}\left(r c_0 - \partial_y f(X_{t+s}^{[n]},y)\right)\mathsf{1}_{\{X_{t+s}^{[n]} \geq b_n(t+s,y)\}}\,\ud s\Big],
\end{aligned}
\end{equation}
for all $(t,y)\in [0,T)\times[0,1]$, with $b_n(T-,y) = \bar{x}(y)$ (cf.\ \eqref{eq::xbarquantity}).
\end{theorem}
\begin{remark}\label{rem:cvf}
A standard approach to obtaining an integral equation for the value function of an OS problem would generally rely on an application of a change of variable formula. The regularity of the value function in OS problems is not sufficient for the use of classical It\^o's formula and various extensions have been considered in the literature (see Subsection 1.1). None of those results is applicable in our case because we only have continuity of $u_n$ and $\partial_x u_n$ but we lack any knowledge of the regularity of $\partial_t u_n$ and $\partial_{xx}u_n$ across the optimal boundary $\partial \cC_y^{[n]}$. Moreover, even inside the continuation set $\cC_y^{[n]}$ we only have weak derivatives $\partial_t u_n$ and $\partial_{xx}u_n$ and 
therefore Peskir's formula (\cite{peskir2005change}) does not directly apply. Instead, we obtain a tailored change of variable formula that extends the one recently obtained in \cite{CDAarXiv2022}.

Notice that results in \cite{CDAarXiv2022} are not applicable to our case because they would require continuity of $\partial_t u_n$ and locally bounded $\partial_{xx}u_n$, whereas we can only rely upon $u \in W^{1,2;p}_{\ell oc}(\mathcal{C}_{y}^{[n]})$. 
\end{remark}

\begin{proof}
Fix $y \in [0,1]$. Let us consider an approximating sequence $(u^{\m}_n)_{\m \geq 1} \subset C^{1,2}([0,T] \times \mathbb{R})$ defined in the following way
\begin{equation}\label{eq:ApproximatingSequence}
    u^{\m}_n(t, x, y):= \m^{2} \int_{(t-\frac{1}{\m})^+}^{t} \int_{x-\frac{1}{\m}}^{x} u_n(s, z, y)\,\ud z\,\ud s.
\end{equation}
For $t\in(0,T]$ and $\m>1/t$,
the derivatives of $u^{\m}_n$ read as:
\begin{align}
    \partial_{t}u^{\m}_n(t, x, y) & = \m^2 \int_{x-\frac{1}{\m}}^{x}\left(u_n(t, z, y)-u_n\left(t-\frac{1}{\m}, z, y\right)\right)\,\ud z.\label{eq:derivativewrttime}\\
    \partial_{x}u^{\m}_n(t, x, y) &= \m^2 \int_{t-\frac{1}{\m}}^{t}\left(u_n(s, x, y)-u_n\left(s, x-\frac{1}{\m},y\right)\right)\,\ud s.\label{eq:firstderivativespace}\\
    \partial_{x x}u^{\m}_n(t, x, y) &= \m^2 \int_{t-\frac{1}{\m}}^{t} \left(\partial_{x} u_n(s, x, y) -\partial_x u_n\left(s, x - \frac{1}{\m}, y\right)\right)\,\ud s.\label{eq:secondderivativespace}
\end{align}

Let us observe that if $(t, x, y) \in [0,T] \times \Sigma$ with $x < b_n(t,y)$ (i.e., $(t,x)\in\cC_y^{[n]}$), then $[0,t] \times (-\infty, x]\subset\cC_y^{[n]}$. Since $u_n(\cdot,y)\in W^{1,2;p}_{\ell oc}(\cC_y^{[n]})$, then 
for each $(t, x) \in \mathcal{C}_{y}^{[n]}$ we can write: 
\begin{align}
    \partial_{t}u^{\m}_n(t, x, y) & = \m^2 \int_{t-\frac{1}{\m}}^{t} \int_{x-\frac{1}{\m}}^{x} \partial_t u_n(s, z, y)\,\ud z \,\ud s\label{eq:derivativewrttimeregularized}\\
    \partial_{x x}u^{\m}_n(t, x, y) & = \m^2 \int_{t-\frac{1}{\m}}^{t} \int_{x-\frac{1}{\m}}^{x} \partial_{xx}u_n(s, z, y)\,\ud z\,\ud s.\label{eq:secondderivativespaceregularized}
\end{align}
In addition, for any compact $K_{y} \subset \mathcal{C}_{y}^{[n]}$ we have:
\begin{equation*}
    \begin{split}
        &\lim_{\m \rightarrow \infty} \sup_{(t, x) \in K_{y}} \left(|u^{\m}_n-u_n|(t, x) + |\partial_x u^{\m}_n -\partial_x u_n|(t, x)\right) = 0,\\
        &\lim_{\m \rightarrow \infty} \partial_{x x}u^{\m}_n(t, x) = \partial_{xx}u_n(t, x),\,a.e.\ (t, x) \in \mathcal{C}_{y}^{[n]},\\
        &\lim_{\m \rightarrow \infty} \partial_t u^{\m}_n(t, x) = \partial_t u_n(t, x),\,a.e.\ (t, x) \in \mathcal{C}_{y}^{[n]}.
    \end{split}
\end{equation*}
We cannot directly write an analogue of \eqref{eq:derivativewrttimeregularized} and \eqref{eq:secondderivativespaceregularized} in the case in which $(t, x, y) \in [0,T] \times \Sigma$ with $x>b_n(t,y)$, due to the lack of regularity across the optimal boundary $b_n$. The procedure we propose in the rest of the proof is valid for every $(t, x, y) \in [0,T]\times\Sigma$.

We now introduce some useful quantities and we make some preliminary observations. First, for fixed $y\in[0,1]$, let
\begin{equation}
    x\mapsto\varphi_n(x, y):=\sup\{t \in [0,T]\,:\,b_n(t,y)>x\},
\end{equation}
be the right continuous (strictly decreasing) inverse of $b_n$ with respect to the time variable $t$. Notice that $x<b_n(t,y)\iff t<\varphi_n(x,y)$ and let $b_{\varepsilon,n} : [0,T] \times [0,1] \rightarrow \mathbb{R}$ be the function defined as:
\begin{equation}\label{eq:definitionbepsilon}
    b_{\varepsilon,n}(t, y) := b_n(t+\varepsilon, y) - \varepsilon.
\end{equation}
Because $b_n$ is non-increasing in the time variable we have
\begin{equation}\label{eq:monotonicityb}
    b_{\varepsilon, n}(t, y) < b_{\varepsilon^{'}, n}(t, y) < b_n(t,y),
\end{equation}
for all $0<\varepsilon'<\varepsilon$. Furthermore, by continuity of $b_n$, the increasing limit holds
\begin{equation}\label{eq:bnconv}
    b_n(t, y)=\lim_{\varepsilon \rightarrow 0}  b_{\varepsilon,n}(t, y),
\end{equation}
uniformly for $(t,y)$ in a compact set. Finally, we introduce the right-continuous (decreasing) inverse of $b_{\eps,n}$ with respect to the time variable as
\begin{equation}\label{eq:varphiepsilon}
   \varphi_{\eps,n}(x,y):=\sup\{t\in[0,T]:b_{\eps,n}(t,y)>x\}. 
\end{equation} 
In particular, we have for every $(x,y) \in \Sigma$ and all $0<\eps'<\eps$, 
\begin{equation*}
   \varphi_{\eps,n}(x,y)<\varphi_{\eps',n}(x,y)<\varphi_n(x,y).
\end{equation*} 
Then, the increasing limit holds
\begin{equation}\label{eq:phiconv}
   \varphi_n(x,y)=\lim_{\varepsilon \rightarrow 0} \varphi_{\eps,n}(x,y).
\end{equation}

At this point, using \eqref{eq:derivativewrttime}, \eqref{eq:firstderivativespace}, and \eqref{eq:secondderivativespace}, which are valid for all $(t, x, y) \in [0,T] \times \Sigma$, we have:
\begin{equation}\label{eq:giulia3}
\begin{aligned} 
        &(\partial_t u^{\m}_n+\frac{\sigma^2(\cdot)}{2}\partial_{xx}u^{\m}_n+a(\cdot,m^{[n-1]}(\cdot))\partial_{x}u^{\m}_n-ru^{\m}_n)(t,x,y)\\
        &=\m^2 \int_{x-\frac{1}{\m}}^{x} \left(u_n\left(t, z, y\right)-u_n\left(t-\frac{1}{\m}, z, y\right)\right)\,\ud z\\
        &\quad+\frac{\sigma^2(x)}{2} \m^2 \int_{t-\frac{1}{\m}}^{t} \left(\partial_{x}u_n\left(s, x, y\right)-\partial_{x}u_n\left(s, x-\frac{1}{\m}, y\right)\right)\,\ud s\\
        &\quad+a(x, m^{[n-1]}(t))\,\m^2\int_{t-\frac{1}{\m}}^{t} \left(u_n(s, x, y)-u_n\left(s, x-\frac{1}{\m}, y\right)\right)\,\ud s\\
        &\quad-r\,\m^2 \int_{t-\frac{1}{\m}}^{t}\int_{x-\frac{1}{\m}}^{x} u_n(s, z, y)\,\ud z \ud s.
\end{aligned}
\end{equation}
Now, we analyse separately the four terms on the right-hand side above. We proceed in four steps.

{\bf Step 1}. 
We rewrite the first integral by considering separately the following three cases:
\begin{equation*}
    t<\varphi_n(z,y),\quad t-\frac{1}{\m}<\varphi_n(z, y) \leq t,\quad\varphi_n(z, y) \leq t-\frac{1}{\m}.
\end{equation*}
Notice that the set $\{z:\varphi_n(z,y)=t\}$ is of zero Lebesgue measure by strict monotonicity of $\varphi_n(\cdot,y)$.
We have
\begin{equation}
\begin{aligned}
    &\m^2 \int_{x-\frac{1}{\m}}^{x} \left(u_n(t, z, y)-u_n\left(t-\frac{1}{\m}, z, y\right)\right)\,\ud z\nonumber\\
    &=\m^2 \int_{x-\frac{1}{\m}}^{x} 1_{\{t<\varphi_n(z,y)\}} \left(u_n(t, z, y)-u_n\left(t-\frac{1}{\m}, z, y\right)\right)\,\ud z\\
    &\quad+ \m^2 \int_{x-\frac{1}{\m}}^{x} 1_{\{t-\frac{1}{\m}<\varphi_n(z,y)<t\}} \left(u_n(t, z, y)-u_n\left(t-\frac{1}{\m}, z, y\right)\right)\,\ud z\\
    &\quad+ \m^2 \int_{x-\frac{1}{\m}}^{x} 1_{\{\varphi_n(z,y)<t-\frac{1}{\m}\}} \left(u_n(t, z, y)-u_n\left(t-\frac{1}{\m}, z, y\right)\right)\,\ud z.
\end{aligned}
\end{equation}
The third integral on the right-hand side is zero because $s>\varphi_n(z,y)\implies (s,z)\in\cC^{[n]}_y$ and
\begin{equation*}
    u_n(t, z, y)-u_n\Big(t-\frac{1}{\m}, z, y\Big)=0,\quad \text{when }\varphi_n(z,y)<t-\frac{1}{\m}.
\end{equation*}
For the second integral we notice that if $t-\frac{1}{\m}<\varphi_n(z,y)<t$, then 
\begin{equation*}
\begin{aligned}
    u_n(t, z, y)-u_n\Big(t-\frac{1}{\m}, z, y\Big)&=c_0-u_n\Big(t-\frac{1}{\m}, z, y\Big)\\
    &=u_n(\varphi_n(z,y), z, y)-u_n\Big(t-\frac{1}{\m}, z, y\Big),
\end{aligned}
\end{equation*}
because, by continuity, $u_n(s,z,y)=c_0$ for $s\in[\varphi_n(z,y),t]$.
Then we can write 
\begin{equation}\label{eq:giulia1}
\begin{aligned}
    &\m^2 \int_{x-\frac{1}{\m}}^{x} \left(u_n(t, z, y)-u_n\left(t-\frac{1}{\m}, z, y\right)\right)\,\ud z\\
    &=\m^2 \int_{x-\frac{1}{\m}}^{x} 1_{\{t<\varphi_n(z,y)\}} \left(u_n(t, z, y)-u_n\left(t-\frac{1}{\m}, z, y\right)\right)\,\ud z\\
    &\quad+ \m^2 \int_{x-\frac{1}{\m}}^{x} 1_{\{t-\frac{1}{\m}<\varphi_n(z,y) \leq t\}} \left(u_n(\varphi_n(z,y), z, y)-u_n\left(t-\frac{1}{\m}, z, y\right)\right)\,\ud z.
\end{aligned}
\end{equation}
In the remainder of this step we are going to show that 
\begin{equation}\label{eq:summary1}
\begin{aligned}
    &\m^2 \int_{x-\frac{1}{\m}}^{x} \left(u_n(t, z, y)-u_n\left(t-\frac{1}{\m}, z, y\right)\right)\,\ud z \\
    &= \lim_{\eps\to0}\m^2 \int_{x-\frac{1}{\m}}^{x} \int_{t-\frac{1}{\m}}^{\left[\varphi_{\varepsilon,n}(z,y) \vee \left(t-\frac{1}{\m}\right) \right] \wedge t}\partial_t u_n(s, z, y)\,\ud s\ud z,
\end{aligned}
\end{equation}
where the function $\varphi_{\varepsilon,n}(z,y)$ is defined in \eqref{eq:varphiepsilon} and the integrand $\partial_t u_n$ is well-defined because $s<\varphi_{\eps,n}(z,y)\implies(s,z)\in\cC^{[n]}_y$.
It is clear by \eqref{eq:phiconv} that:
\begin{equation*}
\begin{aligned}
   &\lim_{\varepsilon \rightarrow 0} \int_{t-\frac{1}{\m}}^{\left[\varphi_{\varepsilon,n}(z,y) \vee \left(t-\frac{1}{\m}\right) \right] \wedge t}\partial_t u_n(s, z, y)\,\ud s\\
   &= u_n\Big(\Big[\varphi_n(z,y) \vee \big(t-\frac{1}{\m}\big) \Big] \wedge t, z, y\Big)-u_n\left(t-\frac{1}{\m}, z, y\right).
\end{aligned}
\end{equation*}
In particular, when $t-\frac{1}{\m}<\varphi_n(z,y)<t$ the right-hand side above reads  
\begin{equation*}
   u_n\big(\varphi_n(z,y), z, y\big)-u_n\Big(t-\frac{1}{\m}, z, y\Big).
\end{equation*}
Moreover, thanks to the increasing limit in \eqref{eq:phiconv} we have, for $z\in\R$,
\begin{equation}
1_{\{t-\frac{1}{\m}<\varphi_n(z,y) \leq t\}}=\lim_{\eps\to0}1_{\{t-\frac{1}{\m}<\varphi_{\eps,n}(z,y)<t\}}\quad\text{and}\quad
1_{\{t<\varphi_n(z,y)\}}=\lim_{\eps\to0}1_{\{t<\varphi_{\eps,n}(z,y)\}}.
\end{equation}
By plugging these expressions into \eqref{eq:giulia1} and using dominated convergence, due to boundedness of $u_n$, we obtain:
\begin{equation*}
\begin{aligned}
    &\m^2 \int_{x-\frac{1}{\m}}^{x} \left(u_n(t, z, y)-u_n\left(t-\frac{1}{\m}, z, y\right)\right)\,\ud z\nonumber\\
    &=\lim_{\eps\to 0}\m^2 \int_{x-\frac{1}{\m}}^{x} 1_{\{t<\varphi_{\eps,n}(z,y)\}} \left(u_n(t, z, y)-u_n\left(t-\frac{1}{\m}, z, y\right)\right)\,\ud z\nonumber\\
    &\quad+ \lim_{\eps\to 0}\m^2 \int_{x-\frac{1}{\m}}^{x} 1_{\{t-\frac{1}{\m}<\varphi_{\eps,n}(z,y)<t\}} \left(u_n(\varphi_{\eps,n}(z,y), z, y)-u_n\left(t-\frac{1}{\m}, z, y\right)\right)\,\ud z\nonumber\\
    &=\lim_{\varepsilon \rightarrow 0} \m^2 \int_{x-\frac{1}{\m}}^{x} 1_{\{t<\varphi_{\varepsilon,n}(z,y)\}}\int_{t-\frac{1}{\m}}^t\partial_t u_n(s, z, y)\,\ud s\,\ud z\nonumber\\
    &\quad+\lim_{\varepsilon \rightarrow 0} \m^2 \int_{x-\frac{1}{\m}}^{x}
    1_{\{t-\frac{1}{\m}<\varphi_{\varepsilon,n}(z,y)<t\}}\int_{t-\frac{1}{\m}}^{\varphi_{\varepsilon,n}(z,y)}\partial_t u_n(s, z, y)\,\ud s\,\ud z\nonumber\\
    &=\lim_{\varepsilon \rightarrow 0} \m^2 \int_{x-\frac{1}{\m}}^{x}1_{\{t<\varphi_{\varepsilon,n}(z,y)\}}\int_{t-\frac{1}{\m}}^{t}
    1_{\{s<\varphi_{\varepsilon,n}(z,y)\}}\partial_t u_n(s, z, y)\,\ud s\,\ud z\nonumber\\
    &\quad+\lim_{\varepsilon \rightarrow 0} \m^2 \int_{x-\frac{1}{\m}}^{x}1_{\{t-\frac{1}{\m}<\varphi_{\varepsilon,n}(z,y)<t\}}\int_{t-\frac{1}{\m}}^{t}1_{\{s<\varphi_{\varepsilon,n}(z,y)\}}\partial_t u_n(s, z, y)\,\ud s\,\ud z
\end{aligned}
\end{equation*}
Now, combining the indicators as $1_{\{t<\varphi_{\varepsilon,n}(z,y)\}}+1_{\{t-\frac{1}{\m}<\varphi_{\varepsilon,n}(z,y)<t\}}=1_{\{t-\frac{1}{\m}<\varphi_{\varepsilon,n}(z,y)\}}$, we obtain
\begin{equation*}
\begin{aligned}
&\m^2 \int_{x-\frac{1}{\m}}^{x} \left(u_n(t, z, y)-u_n\left(t-\frac{1}{\m}, z, y\right)\right)\,\ud z\nonumber\\
    &=\lim_{\varepsilon\rightarrow0}\m^2 \int_{x-\frac{1}{\m}}^{x} 1_{\{t-\frac1\m < \varphi_{\eps,n}(z,y)\}} \int_{t-\frac{1}{\m}}^{t}1_{\{s<\varphi_{\eps,n}(z,y)\}}\partial_{t}u_n(s, z, y)\,\ud s\,\ud z\nonumber\\
    &=\lim_{\eps\to0}\m^2 \int_{x-\frac{1}{\m}}^{x}
    \int_{t-\frac{1}{\m}}^{t}1_{\{s<\varphi_{\eps,n}(z,y)\}}\partial_{t}u_n(s, z, y)\,\ud s\,\ud z,
    \end{aligned}
    \end{equation*}
where the second equality holds because $\{s<\varphi_{\eps,n}(z,y)\}\subset \{t-\frac1\m<\varphi_{\eps,n}(z,y)\}$ for $s\in(t-\frac1\m,t)$. Therefore, \eqref{eq:summary1} holds as claimed.

It is worth noticing that since 
$s<\varphi_{\eps,n}(z,y)\iff z<b_{\eps,n}(s,y)$, by construction, then by Fubini's theorem we can swap the order of integration in the final expression above and obtain also 
\begin{equation*}
\begin{aligned}
&\lim_{\eps\to0}\m^2 \int_{x-\frac{1}{\m}}^{x}
    \int_{t-\frac{1}{\m}}^{t}1_{\{s<\varphi_{\eps,n}(z,y)\}}\partial_{t}u_n(s, z, y)\,\ud s\,\ud z\\
    &=\lim_{\eps\to0}\m^2 \int_{t-\frac{1}{\m}}^{t}\int_{x-\frac{1}{\m}}^{x}
    1_{\{z<b_{\eps,n}(s,y)\}}\partial_{t}u_n(s, z, y)\,\ud z\,\ud s.
\end{aligned}
\end{equation*}

{\bf Step 2}. We now consider the second integral in \eqref{eq:giulia3}. Similarly to what we have done before, we rewrite the integral by considering separately the following cases:
\begin{equation*}
    x<b_n(t,y),\quad x-\frac{1}{\m} < b_n(t, y) \leq x,\quad b_n(t,y) \leq x-\frac{1}{\m}.
\end{equation*}
We have
\begin{equation}
\begin{aligned}
    &\m^2\int_{t-\frac{1}{\m}}^{t} \left(\partial_x u_n(s, x, y)-\partial_x u_n\left(s, x-\frac{1}{\m}, y\right)\right)\,\ud s\nonumber\\
    &=\m^2\int_{t-\frac{1}{\m}}^{t} 1_{\{x<b_n(s,y)\}}\left(\partial_x u_n(s, x, y)-\partial_x u_n\left(s, x-\frac{1}{\m}, y\right)\right)\,\ud s\label{eq:differencederivativeOne}\\
    &\quad+\m^2\int_{t-\frac{1}{\m}}^{t} 1_{\{x-\frac{1}{\m}<b_n(s,y)\leq x\}}\left(\partial_x u_n(s, x, y)-\partial_x u_n\left(s, x-\frac{1}{\m}, y\right)\right)\,\ud s\\
    &\quad+\m^2\int_{t-\frac{1}{\m}}^{t} 1_{\{b_n(s,y)\leq x-\frac{1}{\m}\}}\left(\partial_x u_n(s, x, y)-\partial_x u_n\left(s, x-\frac{1}{\m}, y\right)\right)\,\ud s.
\end{aligned}
\end{equation}
The final integral is zero because $x-1/\m\ge b_n(s,y)$ implies $(s,x-1/\m)\in\cS^{[n]}_y$ and $(s,x)\in\cS^{[n]}_y$.
The penultimate integral instead reduces to 
\begin{equation}\label{eq:differencederivativeTwoOne}
\begin{aligned}
    &-\m^2\int_{t-\frac{1}{\m}}^{t}1_{\{x-\frac{1}{\m}<b_n(s,y)\leq x\}}\partial_x u_n\left(s, x-\frac{1}{\m}, y\right)\,\ud s\\
     &=\m^2\int_{t-\frac{1}{\m}}^{t}1_{\{x-\frac{1}{\m}<b_n(s,y)\leq x\}}\Big(\partial_x u_n (s,b_n(s,y),y)-\partial_x u_n\left(s, x-\frac{1}{\m}, y\right)\Big)\,\ud s,
\end{aligned}
\end{equation}
because $u_n(s,x,y)=c_0$ and $\partial_x u_n(s,x,y)=0$ for $x-\frac{1}{\m} < b_n(s,y) \leq x$, by smooth-fit. Then we are left with 
\begin{equation}\label{eq:giulia2}
\begin{aligned}
&\m^2\int_{t-\frac{1}{\m}}^{t} \left(\partial_x u_n(s, x, y)-\partial_x u_n\left(s, x-\frac{1}{\m}, y\right)\right)\,\ud s\\
&=\m^2\int_{t-\frac{1}{\m}}^{t} 1_{\{x<b_n(s,y)\}}\left(\partial_x u_n(s, x, y)-\partial_x u_n\left(s, x-\frac{1}{\m}, y\right)\right)\,\ud s\\
&\quad+\m^2\int_{t-\frac{1}{\m}}^{t} 1_{\{x-\frac{1}{\m}<b_n(s,y)\leq x\}}\left(\partial_x u_n(s, b_n(s,y), y)-\partial_x u_n\left(s, x-\frac{1}{\m}, y\right)\right)\,\ud s.
\end{aligned}
\end{equation}
In the remainder of this step we are going to show that 
\begin{equation}\label{eq:summary2}
\begin{aligned}
&\m^2\int_{t-\frac{1}{\m}}^{t} \left(\partial_x u_n(s, x, y)-\partial_x u_n\left(s, x-\frac{1}{\m}, y\right)\right)\,\ud s\\
&= \lim_{\varepsilon\rightarrow 0}\int_{t-\frac{1}{\m}}^{t} \int_{x-\frac{1}{\m}}^{[b_{\varepsilon,n}(s,y) \wedge x] \vee \left(x-\frac{1}{\m}\right)} \partial_{xx}u_n(s, z, y)\,\ud z\,\ud s,
\end{aligned}
\end{equation}
where $b_{\varepsilon,n}(t,y)$ is defined in \eqref{eq:definitionbepsilon} and the second order derivative under the integral is well-defined because $z<b_{\eps,n}(s,y)\implies(s,z)\in\cC^{[n]}_y$.

It is clear that for $x-1/\m<b_n(s,y)$,
\begin{equation*}
\partial_x u_n(s, b_n(s,y), y)-\partial_x u_n\left(s, x-\frac{1}{\m}, y\right)= \lim_{\varepsilon\rightarrow0}    \int_{x-\frac{1}{\m}}^{b_{\varepsilon,n}(s,y) } \partial_{xx}u_n(s, z, y)\,\ud z.
\end{equation*}
Moreover, for all $s\in[t-\frac1\m,t]$, we have
\begin{equation}\label{eq:madda2}
    1_{\{x<b_n(s,y)\}}=\lim_{\varepsilon \rightarrow 0}1_{\{x<b_{\varepsilon,n}(s,y)\}}\quad \text{and}\quad 1_{\{x-\frac{1}{\m}<b_n(s,y)\leq x\}} = \lim_{\varepsilon\rightarrow 0}1_{\{x-\frac{1}{\m}<b_{\varepsilon,n}(s,y)\leq x\}}.
\end{equation}
Both limits follow by the pointwise, increasing convergence of $b_{\varepsilon, n}$ to $b_n$ in \eqref{eq:bnconv}. 
By plugging these expressions into \eqref{eq:giulia2}, we obtain: 
\begin{equation*}
\begin{aligned}
    &\m^2 \int_{t-\frac{1}{\m}}^{t}\left(\partial_x u_n(s, x, y)-\partial_x u_n\left(s, x-\frac{1}{\m}, y\right)\right)\,\ud s\\
    &=\lim_{\varepsilon\rightarrow0} \m^2\int_{t-\frac{1}{\m}}^{t} 1_{\{x<b_{\varepsilon,n}(s,y)\}}\int_{x-\frac{1}{\m}}^{x}\partial_{x x}u_n(s, z, y)\,\ud z\,\ud s\\
&\quad+\lim_{\varepsilon\rightarrow0}\m^2\int_{t-\frac{1}{\m}}^{t}1_{\{x-\frac{1}{\m}<b_{\varepsilon,n}(s,y)\leq x\}}\int_{x-\frac{1}{\m}}^{b_{\varepsilon,n}(s,y)}\partial_{x x}u_n(s, z, y)\,\ud z\,\ud s\\
    &=\lim_{\varepsilon\rightarrow0} \m^2\int_{t-\frac{1}{\m}}^{t} 1_{\{x<b_{\varepsilon,n}(s,y)\}}\int_{x-\frac{1}{\m}}^{x}1_{\{z<b_{\varepsilon,n}(s,y)\}}\partial_{x x}u_n(s, z, y)\,\ud z\,\ud s\nonumber\\
    &\quad+\lim_{\varepsilon\rightarrow0}\m^2\int_{t-\frac{1}{\m}}^{t}1_{\{x-\frac{1}{\m}<b_{\varepsilon,n}(s,y)\leq x\}}\int_{x-\frac{1}{\m}}^{x}1_{\{z<b_{\varepsilon,n}(s,y)\}}\partial_{x x}u_n(s, z, y)\,\ud z\,\ud s.
\end{aligned}
\end{equation*}
Now, recombining the indicator functions $1_{\{x<b_{\varepsilon,n}(s,y)\}}+1_{\{x-\frac{1}{\m}<b_{\varepsilon,n}(s,y)\leq x\}}=1_{\{x-\frac{1}{\m}<b_{\varepsilon,n}(s,y)\}}$ and noticing that, for $z\in(x-\frac1\m,x)$, 
$1_{\{x-\frac{1}{\m}<b_{\varepsilon,n}(s,y)\}}1_{\{z<b_{\varepsilon,n}(s,y)\}}=1_{\{x-\frac1\m <z<b_{\varepsilon,n}(s,y)\}}=1_{\{z<b_{\varepsilon,n}(s,y)\}}$, 
we obtain
\begin{equation*}
   \begin{aligned}
   \m^2\! \int_{t-\frac{1}{\m}}^{t}\!\!\left(\partial_x u_n(s, x, y)\!-\!\partial_x u_n\left(s, x\!-\!\frac{1}{\m}, y\right)\right)\,\ud s=\lim_{\varepsilon\rightarrow0}\m^2\!\int_{t-\frac{1}{\m}}^{t}\int_{x-\frac{1}{\m}}^{x}\!\!1_{\{z<b_{\varepsilon,n}(s,y)\}}\partial_{x x}u_n(s, z, y)\,\ud z\,\ud s.
    \end{aligned}
    \end{equation*}
Therefore, \eqref{eq:summary2} holds as claimed.
\vspace{+5pt}

{\bf Step 3}. Finally, we consider the third integral in \eqref{eq:giulia3}. By analogous arguments as those used in Step 2 and noticing that
\[
u_n(s, x, y)-u_n\left(s, x-\frac{1}{\m}, y\right)=0, \quad\text{for $b_n(s,y)\le x-\frac1\m$},
\]
we have
\begin{equation*}
\begin{aligned}
    &\m^2 \int_{t-\frac{1}{\m}}^{t}\left( u_n(s, x, y)-u_n\left(s, x-\frac{1}{\m}, y\right)\right)\,\ud s\\
    &=\m^2\int_{t-\frac{1}{\m}}^{t}1_{\{x<b_n(s,y)\}}\left(u_n(s, x, y)-u_n\left(s, x-\frac{1}{\m}, y\right)\right)\,\ud s\nonumber\\
    &\quad+\m^2\int_{t-\frac{1}{\m}}^{t}1_{\{x-\frac{1}{\m}<b_n(s,y)\leq x\}}\left(u_n(s, x, y)-u_n\left(s, x-\frac{1}{\m}, y\right)\right)\,\ud s\nonumber\\
    &=\lim_{\varepsilon\rightarrow0} \m^2\int_{t-\frac{1}{\m}}^{t} 1_{\{x<b_{\varepsilon,n}(s,y)\}}\int_{x-\frac{1}{\m}}^{x} \partial_{x }u_n(s, z, y)\,\ud z\,\ud s\nonumber\\
&\quad+\lim_{\varepsilon\rightarrow0}\m^2\int_{t-\frac{1}{\m}}^{t}1_{\{x-\frac{1}{\m}<b_{\varepsilon,n}(s,y)\leq x\}}\int_{x-\frac{1}{\m}}^{b_{\varepsilon,n}(s,y)}\partial_{x }u_n(s, z, y)\,\ud z\,\ud s\\
&=\lim_{\varepsilon\rightarrow0} \m^2 \int_{t-\frac{1}{\m}}^{t}
    1_{\{x-\frac{1}{\m}<b_{\varepsilon,n}(s,y)\}}\int_{x-\frac{1}{\m}}^{x}
    1_{\{z<b_{\varepsilon,n}(s,z)\}}\partial_{x}u_n(s, z, y)\,\ud z\,\ud s\nonumber\\ &=\lim_{\varepsilon\rightarrow0}\m^2\int_{t-\frac{1}{\m}}^{t}\int_{x-\frac{1}{\m}}^{x}1_{\{z<b_{\varepsilon,n}(s,y)\}}\partial_{x}u_n(s, z, y)\,\ud z\,\ud s.
   \end{aligned}
\end{equation*}

{\bf Step 4}. Now we combine the expressions from the previous three steps into \eqref{eq:giulia3}. For $(t,x,y) \in [0,T] \times \Sigma$ we get
\begin{equation}
\begin{aligned}\label{eq::ito1}
        &(\partial_t u^{\m}_n+\frac{\sigma^2(\cdot)}{2}\partial_{xx}u^{\m}_n+a(\cdot,m^{[n-1]}(\cdot))\partial_{x}u^{\m}_n-ru^{\m}_n)(t,x,y)\\
        &=\lim_{\varepsilon\rightarrow0}\m^2\int_{t-\frac{1}{\m}}^{t}\!\!\int_{x-\frac{1}{\m}}^{x}\!\!\!  1_{\{z<b_{\varepsilon,n}(s,y)\}}\!\left(\!\partial_t u_n\!+\!\frac{\sigma^2(x)}{2}\partial_{xx}u_n\!+\!a(x,m^{[n-1]}(t))\partial_x u_n\! -\!r u_n\!\right)\!(s, z, y)\,\ud z\,\ud s\\
        &\quad-r \m^2 \int_{t-\frac{1}{\m}}^{t}\int_{x-\frac{1}{\m}}^{x} 1_{\{z \geq b_n(s,y)\}}u_n(s, z, y)\,\ud z\,\ud s.
\end{aligned}
\end{equation}
       First we notice that 
        \[
      \m^2 \int_{t-\frac{1}{\m}}^{t}\int_{x-\frac{1}{\m}}^{x} 1_{\{z \geq b_n(s,y)\}}u_n(s, z, y)\,\ud z\,\ud s=\m^2 \int_{t-\frac{1}{\m}}^{t}\int_{x-\frac{1}{\m}}^{x} 1_{\{z \geq b_n(s,y)\}}c_0\,\ud z\,\ud s.  
        \]
Then we notice that the coefficients $\sigma^2(x)$ and $a(m^{[n-1]}(t),x)$ inside the first integral on the right-hand side of \eqref{eq::ito1} depend on $(t,x)$ whereas we would like them to depend on $(s,y)$ in order to exploit the PDE for $u_n$ (cf.\ \eqref{eq:weak_PDE}). Then, adding and subtracting terms we continue from the right-hand side of \eqref{eq::ito1} with
\begin{equation}\label{eq:pdeun}
\begin{aligned}
&\lim_{\varepsilon\rightarrow0}\m^2 \int_{t-\frac{1}{\m}}^{t}\int_{x-\frac{1}{\m}}^{x} 1_{\{z<b_{\varepsilon,n}(s,y)\}}\left(\partial_t u_n+\frac{\sigma^2(\,\cdot\,)}{2}\partial_{xx}u_n+a(\cdot, m^{[n-1]}(\,\cdot\,))\partial_x u_n -r u_n\right)(s, z, y)\,\ud z\,\ud s\\   &\quad+\lim_{\varepsilon\rightarrow0}\m^2
\int_{t-\frac{1}{\m}}^{t}\int_{x-\frac{1}{\m}}^{x} 1_{\{z<b_{\varepsilon,n}(s,y)\}}\frac{\sigma^2(x)-\sigma^2(z)}{2}\partial_{xx}u_n(s, z, y)\,\ud z\,\ud s\\    &\quad+\lim_{\varepsilon\rightarrow0}\m^2
\int_{x-\frac{1}{\m}}^{x}\int_{t-\frac{1}{\m}}^{t} 1_{\{z<b_{\varepsilon,n}(s,y)\}}[a(x,m^{[n-1]}(t))-a(z,m^{[n-1]}(s))]\partial_x u_n(s, z, y)\,\ud s\,\ud z\\
&\quad-r\m^2\int_{t-\frac{1}{\m}}^{t}\int_{x-\frac{1}{\m}}^{x} 1_{\{z \geq b_n(s,y)\}}c_0\,\ud z\,\ud s\\
&=-\lim_{\varepsilon\rightarrow0}\m^2 \int_{t-\frac{1}{\m}}^{t}\int_{x-\frac{1}{\m}}^{x} \left(1_{\{z<b_{\varepsilon,n}(s,y)\}}\partial_{y}f(z,y)+1_{\{z\geq b_{\varepsilon,n}(s,y)\}}r c_0\right) \,\ud z\,\ud s +\lim_{\varepsilon\rightarrow0}F_{\varepsilon,n}^{\m}(t,x),
        \end{aligned}
\end{equation}
where for the equality we used that $u_n$ solves \eqref{eq:weak_PDE} in $\{z<b_n(s,y)\}$ and we introduced:
\begin{equation*}
\begin{aligned}
    F_{\varepsilon,n}^{\m}(t,x)&:=\m^2\int_{t-\frac{1}{\m}}^{t}\int_{x-\frac{1}{\m}}^{x}1_{\{z<b_{\varepsilon,n}(s,y)\}}\frac{\sigma^2(x)-\sigma^2(z)}{2}\partial_{xx}u_n(s, z, y)\,\ud z\,\ud s\\
         &\quad+\m^2\int_{t-\frac{1}{\m}}^{t}\int_{x-\frac{1}{\m}}^{x} 1_{\{z<b_{\varepsilon,n}(s,y)\}}\left[a(x,m^{[n-1]}(t))-a(z,m^{[n-1]}(s))\right]\partial_x u_n(s, z, y)\,\ud z\,\ud s.
\end{aligned}
\end{equation*}
Of course $F^\m_{\eps,n}(t,x)$ also depends on $y$, but since the latter variable is fixed, we avoid a heavier notation.

To show that we can pass to the limit in $\eps$, first, and then in $\m$, it is convenient to rewrite the first integral above removing the second-order derivative from $u_n$, which we cannot control.  Integrating by parts and using the fact that $\sigma^2(\cdot)\in C^{1}(\mathbb{R})$ we obtain:
\begin{equation}\label{eq:madda3}
\begin{aligned}
        &\m^{2} \int_{t-\frac{1}{\m}}^{t}\int_{x-\frac{1}{\m}}^{x} 1_{\{z<b_{\eps,n}(s,y)\}}\frac{\sigma^2(x)-\sigma^2(z)}{2}\,\partial_{xx}u_n(s, z, y)\,\ud z\,\ud s\\
        &= \m^2 \int_{t-\frac{1}{\m}}^{t}\frac{\sigma^2(x)-\sigma^2(x\wedge b_{\eps,n}(s,y))}{2}\partial_x u_n(s, x \wedge b_{\eps,n}(s,y),y)\,\ud s\\
        &\quad-\m^{2} \int_{t-\frac{1}{\m}}^{t}\frac{\sigma^2(x)-\sigma^2\left(\left(x-\frac{1}{\m}\right)  \wedge
        b_{\eps,n}(s,y)\right)}{2}\partial_x u_n\left(s, \left(x-\frac{1}{\m}\right)  \wedge b_{\eps,n}(s,y), y\right)\,\ud s\\
        &\quad+\m^{2} \int_{t-\frac{1}{\m}}^{t}\int_{x-\frac{1}{\m}}^{x}1_{\{z<b_{\eps,n}(s,y)\}}\sigma(z)\partial_x\sigma(z)\partial_x u_n(s, z, y)\,\ud z\,\ud s.
\end{aligned}
\end{equation}
Then, plugging \eqref{eq:madda3} into the expression for $F^\m_{\eps,n}(t,x)$ we see that we can let $\eps\to 0$ and use continuity of $\sigma$, $\partial_x\sigma$, $\partial_x u_n$ and \eqref{eq:madda2} to obtain 
\begin{equation}\label{eq::effeennefirst}
\begin{aligned}
    &\lim_{\eps\to 0}F^{\m}_{\eps,n}(t,x)\\
    &= -\m^{2} \int_{t-\frac{1}{\m}}^{t}\frac{\sigma^2(x)-\sigma^2\left(\left(x-\frac{1}{\m}\right)  \wedge
        b_{n}(s,y)\right)}{2}\partial_x u_n\left(s, \left(x-\frac{1}{\m}\right) \wedge b_{n}(s,y), y\right)\,\ud s \\
    &\quad+\m^{2} \int_{t-\frac{1}{\m}}^{t}\int_{x-\frac{1}{\m}}^{x}1_{\{z<b_{n}(s,y)\}}\sigma(z)\partial_x\sigma(z)\partial_x u_n(s, z, y)\,\ud z\,\ud s\\
     &\quad+\m^2\int_{t-\frac{1}{\m}}^{t}\int_{x-\frac{1}{\m}}^{x}1_{\{z<b_n(s, y)\}}\left[a(x,m^{[n-1]}(t))-a(z,m^{[n-1]}(s))\right]\partial_x u_n(s, z, y)\,\ud z\,\ud s\\
     &\eqqcolon F^\m_{n}(t,x),
\end{aligned}
\end{equation}
where, in particular, the first integral in \eqref{eq:madda3} vanishes because
\[
\m^2 \int_{t-\frac{1}{\m}}^{t}\frac{\sigma^2(x)-\sigma^2(x\wedge b_{n}(s,y))}{2}\partial_x u_n(s, x \wedge b_{n}(s,y),y)\,\ud s=0,
\]
due to $\partial_x u_n(s,z,y)=0$ for $z\ge b_n(s,y)$. Notice that the first two integrals in the expression for $F^\m_n(t,x)$ recombine to yield
\begin{equation}\label{eq:star}
\begin{aligned}
    \m^{2} \int_{t-\frac{1}{\m}}^{t}\int_{x-\frac{1}{\m}}^{x} 1_{\{z<b_{n}(s,y)\}}\frac{\sigma^2(x)-\sigma^2(z)}{2}\,\partial_{xx}u_n(s, z, y)\,\ud z\,\ud s.
\end{aligned}
\end{equation}
It is worth noticing for later use that \eqref{eq:pdeun} then reads as
\begin{equation}\label{eq:pdeun2}
\begin{aligned}
&(\partial_t u^{\m}_n+\frac{\sigma^2(\cdot)}{2}\partial_{xx}u^{\m}_n+a(\cdot,m^{[n-1]}(\cdot))\partial_{x}u^{\m}_n-ru^{\m}_n)(t,x,y)\\
&=-\m^2 \int_{t-\frac{1}{\m}}^{t}\int_{x-\frac{1}{\m}}^{x} \left(1_{\{z<b_{n}(s,y)\}}\partial_{y}f(z,y)+1_{\{z\geq b_{n}(s,y)\}}r c_0\right) \,\ud z\,\ud s +F_{n}^{\m}(t,x).
\end{aligned}
\end{equation}

At this point, we show that for any compact $K\subset[0,T]\times\R$
\begin{equation}\label{eq:boundeffeepsilon}
    |F^{\m}_{n}(t,x)| \leq c_{K},\quad \forall \m\in\N,
\end{equation}
where $c_{K}>0$ is a constant depending only on the compact (and possibly on the fixed value of $y\in[0,1]$).   
Recalling Proposition \ref{prop:continuityspatialderivative} and Assumption \ref{itm:A1}-(ii) we set $M_{K}:=\sup_{(s,z)\in K}|\partial_x u_n(s,z,y)|$, $S_K:=\sup_{(s,z) \in K}|\sigma(z)|$, and $S_K':=\sup_{(s,z) \in K}|\partial_x\sigma(z)|$, where we slightly abuse notation by considering $\sigma$ and $\partial_x\sigma$ as functions of two variables. For the second integral in \eqref{eq::effeennefirst} we have:
\begin{equation*}
        \left|\m^{2} \int_{t-\frac{1}{\m}}^{t}\int_{x-\frac{1}{\m}}^{x}1_{\{z<b_n(s,y)\}}\sigma(z)\partial_x\sigma(z)\partial_x u_n(s, z, y)\,\ud z\,\ud s\right|\leq S_K S^{'}_{K} M_{K}:=C_K.
\end{equation*}
For the third integral in \eqref{eq::effeennefirst}, we use Assumption \ref{itm:A1}--(i) and Proposition \ref{prop:continuityspatialderivative} to conclude that
\begin{equation}
\begin{aligned}\label{eq:final1}
     &\left|  \m^{2}\int_{t-\frac{1}{\m}}^{t}\int_{x-\frac{1}{\m}}^{x}1_{\{z<b_n(s,y)\}}[a(x, m^{[n-1]}(t))-a(z, m^{[n-1]}(s))]\,\partial_x u_n(s, z, y)\,\ud z\,\ud s\right|\nonumber\\
     &\leq \m^{2} M_K\int_{t-\frac{1}{\m}}^{t}\int_{x-\frac{1}{\m}}^{x}( L(a) (|x-z|+|m^{[n-1]}(t)-m^{[n-1]}(s)|)\,\ud z\,\ud s\nonumber\\
     &\leq \frac{M_K}{\m}L(a)+2 M_K L(a):=C'_K,
\end{aligned}
\end{equation}
upon using that $|x-z|\leq 1/\m$ and $|m^{[n-1]}(\cdot)|\le 1$. Finally, for the first integral in the formula for $F^\m_n(t,x)$ we only need to consider the case $b_n(s,y)>x-\frac{1}{\m}$, because otherwise the integrand vanishes thanks to smooth-fit. Then, 
\begin{equation}
\begin{aligned}
& \left| \m^{2} \int_{t-\frac{1}{\m}}^{t}1_{\{b_{n}(s,y)>x-\frac{1}{\m}\}}   \frac{\sigma^2(x)-\sigma^2\left(x-\frac{1}{\m}\right)}{2}\partial_x u_n\left(s, x-\frac{1}{\m}, y\right)\,\ud s\right|\\
&\leq \m \frac{2S_K S_K^{'}}{2 \m} M_K = S_K S_K^{'} M_{K}:=C^{''}_{K},
\end{aligned}
\end{equation}
upon using that $\sigma^{2}$ is Lipschitz with constant bounded by $2 S_{K} S^{'}_{K}$. The calculations above show that the bound in \eqref{eq:boundeffeepsilon} holds true with $c_K:=C_K+C'_K+C^{''}_K$.

Next, we study the limit of $F^{\m}_{n}(t, x)$ as $\m \rightarrow \infty$. We start by noticing that limit is well-defined for $x\neq b_n(t,y)$: for $x>b_n(t,y)$, taking $\m$ sufficiently large we have $x-\frac1\m>b_n(s,y)$ for all $s\in[t-\frac1\m,t]$ by continuity and therefore 
$
\lim_{\m\to\infty} F^\m_{n}(t,x)=0
$;
instead, for  $x<b_n(t,y)$ it is convenient to undo the integration by parts in 
\eqref{eq::effeennefirst} and, using \eqref{eq:star}, rewrite 
\begin{equation*}
\begin{aligned}
F^\m_{n}(t,x)&=\m^2\int_{t-\frac1\m}^t\int_{x-\frac1\m}^x\frac{\sigma^2(x)-\sigma^2(z)}{2}\partial_{xx}u_n(s,z,y)1_{\{z<b_n(s,y)\}}\ud z\ud s\\
&\quad+\m^2\int_{t-\frac{1}{\m}}^{t}\int_{x-\frac{1}{\m}}^{x}1_{\{z<b_n(s, y)\}}\left[a(x,m^{[n-1]}(t))-a(z,m^{[n-1]}(s))\right]\partial_x u_n(s, z, y)\,\ud z\,\ud s.
\end{aligned}
\end{equation*}
Since $\sigma$ is Lipschitz and $\partial_{xx}u_n\in L^p_{\ell oc}(\cC_y)$, for $p>1$, the first term in the formula above can be upper bounded with
\begin{equation*}
\begin{aligned}
&\m^2\int_{t-\frac1\m}^t\int_{x-\frac1\m}^x\Big|\frac{\sigma^2(x)-\sigma^2(z)}{2}\partial_{xx}u_n(s,z,y)1_{\{z<b_n(s,y)\}}\Big|\ud z\ud s\\
&\le S_K S'_K\frac{1}{\m} \m^2 \Big(\frac{1}{\m}\Big)^{2-\frac2p}\big\|\partial_{xx}u_n(\cdot,y)\big\|_{L^p([0,t]\times[x-1,x])},
\end{aligned}
\end{equation*}
where we used  $[0,t]\times(-\infty,x]\times\{y\}\subset\cC_y$ for $x<b_n(t,y)$, combined with H\"older inequality. For $p>2$, the term on the right-hand side above vanishes when $\m\to\infty$.
Instead, the second term does not vanish and it yields
\begin{equation*}
\lim_{\m\rightarrow\infty} F^{\m}_{n}(t, x) = [a(x, m^{[n-1]}(t))-a(x, m^{[n-1]}(t-))]\partial_x u_n(t, x, y),
\end{equation*}
by the fundamental theorem of calculus.
The latter expression is equal to zero for $t\in[0,T]\setminus J_m$, where $J_m$ is the countable collection of jump times of $m^{[n-1]}(t)$.
In conclusion we have shown that 
\begin{equation}\label{eq:limFm}
\lim_{\m\to\infty} F^\m_{n}(t,x)=0,\quad \text{for $(t,x)\in[0,T]\times\R$ such that $x\neq b_n(t,y)$ and $t\notin J_m$}.
\end{equation}

Next we are going to use the results above to derive a formula which is the analogue of Dynkin's formula.
Let $(K^{(\ell)})_{\ell \in \mathbb{N}}$ be a sequence of compact sets increasing to $\mathbb{R}$ and define
\begin{equation*}
    \tau^{\ell}_{n}(t, x) := \inf\{s \geq 0 : X_{t+s}^{[n];t,x} \notin K^{(\ell)}\} \wedge (T-t).
\end{equation*}
In addition, let $(u^{\m}_n)_{\m \in \mathbb{N}}$ be the approximating sequence defined in \eqref{eq:ApproximatingSequence}. To simplify notation we let $\cL$ denote the second order differential operator
$\cL:= \frac{\sigma^2}{2}\partial_{xx}+a\partial_x-r$.
For each $\ell \in \mathbb{N}$, we can apply It\^o's lemma to obtain
\begin{equation}\label{eq::StepOne1}
    \begin{aligned}
        \e^{-r (s\wedge\tau^{\ell}_n)}u^{\m}_n(t+s\wedge\tau^{\ell}_n, X_{t+s\wedge\tau^{\ell}_n}^{[n];t, x}, y)
        &= u^{\m}_n(t, x, y) + \int_{0}^{s\wedge\tau^{\ell}_n} \e^{-r u}\mathcal{L}u^{\m}_n(t+u, X_{t+u}^{[n];t, x}, y)\,\ud u\\
        &\quad+ \int_{0}^{s\wedge\tau^{\ell}_n} \e^{-ru}\partial_{x}u^{\m}_n(t+u, X_{t+u}^{[n];t, x}, y)\sigma(X_{t+u}^{[n];t, x})\ud W_{t+u},
    \end{aligned}
\end{equation}
for all $(t, x, y)\in[0,T]\times\Sigma$ and $s \in [0, T-t]$. In particular, using \eqref{eq:pdeun2}
\begin{equation}
    \begin{aligned}
        &\mathcal{L}u^{\m}_n(t+u, X_{t+u}^{[n];t, x}, y)=\left(\partial_t u^{\m}_n + \frac{\sigma(\,\cdot\,)}{2}\partial_{x x}u^{\m}_n + a(\,\cdot\,,m^{[n-1]}(\,\cdot\,))\partial_{x}u^{\m}_n - r u^{\m}_n\right)(t+u, X_{t+u}^{[n];t, x}, y)\\
        &=-\m^2 \int_{t+u-\frac{1}{\m}}^{t+u}\int_{X_{t+u}^{[n];t,x}-\frac{1}{\m}}^{X_{t+u}^{[n];t,x}} \left(1_{\{z<b_n(v,y)\}}\partial_{y}f(z,y)+1_{\{z\geq b_n(v,y)\}}r c_0\right) \,\ud z\,\ud v +F_{n}^{\m}\big(t+u,X_{t+u}^{[n];t,x}\big).
    \end{aligned}
\end{equation}
Therefore, \eqref{eq::StepOne1} becomes:
\begin{equation*}
\begin{split}
    &\e^{-r (s\wedge\tau^{\ell}_n)}u^{\m}_n(t+s\wedge\tau^{\ell}_n, X_{t+s\wedge\tau^{\ell}_n}^{[n];t,x}, y)\\
    &=u^{\m}_n(t, x, y)\\
    &\quad- \m^2\int_{0}^{s\wedge\tau^{\ell}_n} \e^{-r u} \int_{t+u-\frac{1}{\m}}^{t+u}\int_{X_{t+u}^{[n];t,x}-\frac{1}{\m}}^{X_{t+u}^{[n];t,x}} \left(1_{\{z<b_n(v,y)\}}\partial_{y}f(z,y)+1_{\{z\geq b_n(v,y)\}}r c_0\right) \,\ud z\,\ud v\,\ud u\\
    &\quad+ \int_{0}^{s\wedge\tau^{\ell}_n} \e^{-r u}F_{n}^{\m}\big(t+u,X_{t+u}^{[n];t,x}\big)\,\ud u+ \int_{0}^{s\wedge\tau^{\ell}_n} \e^{-ru}\partial_{x}u^{\m}_n(t+u, X_{t+u}^{[n];t, x}, y)\sigma(X_{t+u}^{[n];t, x})\ud W_{t+u}.
\end{split}
\end{equation*}

Now we take expectation on both sides. We use that $\partial_{x} u(\cdot,\cdot,y)$ is bounded on $[0,T]\times K^{(\ell)}$ and pick $s=T-t$. That yields
\begin{equation*}
\begin{aligned}
    u^{\m}_n(t, x, y) &= \E\left[\e^{-r \tau^{(\ell)}_{n}}u^{\m}_n(t+\tau^{(\ell)}_{n}, X_{t+\tau^{(\ell)}_{n}}^{t,x}, y)\right]\\
    &\quad + \m^2\E\left[\int_{0}^{\tau^\ell_n} \e^{-r u} \int_{t+u-\frac{1}{\m}}^{t+u}\int_{X_{t+u}^{[n];t,x}-\frac{1}{\m}}^{X_{t+u}^{[n];t,x}} \left(1_{\{z<b_n(v,y)\}}\partial_{y}f(z,y)+1_{\{z\geq b_n(v,y)\}}r c_0\right) \,\ud z\,\ud v\,\ud u\right]\\
    &\quad-\E\left[\int_0^{s\wedge\tau^\ell_n}\e^{-r u} F_{n}^{\m}(t+u,X^{[n];t,x}_{t+u})\,\ud u\right].
\end{aligned}
\end{equation*}
Since the process $X_{t+\cdot \wedge \tau^{(\ell)}_{n}}^{t,x}$ is bound to evolve in $K^{(\ell)}$, we can use dominated convergence to take limits as $\m\rightarrow\infty$ on both sides of the previous equation. Because $s\mapsto b_n(t+s,y)$ is of bounded variation and $\sigma$ is continuous with $\sigma(x)>0$, $x\in\R$, then it is well known that (cf., e.g., \cite[Lemma 4.3]{KP} or \cite[Lemma A.1]{DGV}, for a statement directly applicable to our setup)
\begin{equation}\label{eq:probab}
\int_0^{T-t}\P\big(X_{t+s}^{[n];t,x}=b_n(t+s,y)\big)\ud s=0.
\end{equation}
Thus, we can use \eqref{eq:boundeffeepsilon}, \eqref{eq:limFm} and dominated convergence theorem to get 
\begin{equation*}
\begin{aligned}
    u_n(t, x, y) &=\E\Big[\e^{-r \tau^{(\ell)}_{n}}u_n\big(t+\tau^{(\ell)}_{n}, X_{t+\tau^{(\ell)}_{n}}^{t,x}, y\big)\Big]\\
&\quad+\E\Big[\int_{0}^{\tau^{(\ell)}_{n}}\e^{-r u}\left(1_{\{X_{t+s}^{[n];t,x}<b_n(s,y)\}}\partial_{y}f(X_{t+s}^{[n];t,x}, y) + 1_{\{X_{t+s}^{[n];t,x} \geq b_n(s,y)\}}r c_0 \right)\,\ud s \Big].
   \end{aligned}
\end{equation*}
Letting $\ell \uparrow \infty$ and observing that $\tau^{(\ell)}_{n} \uparrow (T-t)$ we can use dominated convergence once again (cf.\ Assumption \ref{itm:A3}-(i)), continuity of $u_n$ and $u_n(T,x)=c_0$ to obtain
\begin{equation*}
\begin{aligned}
u_n(t, x, y) &=\e^{-r (T-t)}c_0\\
&\quad+\E\Big[\int_{0}^{T-t}\e^{-r s}\left(\partial_{y}f(X_{t+s}^{[n];t, x}, y) + \left(r c_0- \partial_{y}f(X_{t+s}^{[n];t, x}, y)\right)1_{\{X_{t+s}^{[n];t,x}\geq b_n(s,y)\}}\right)\,\ud s\Big].
\end{aligned}
\end{equation*}
This proves \eqref{eq:integral_rep}. Evaluating the above expression at $x=b_n(t,y)$ yields \eqref{eq:integral_eq0} and concludes the proof of Theorem \ref{prop:Integralequation}.
\end{proof}

It is shown in \cite[Lemma 3.7 and Corollary 3.14]{CDAGLAAP2021} that $X^*_{t+s}:= \lim_{n\to\infty} X^{[n]}_{t+s}$ is well-defined $\P$-a.s.\ and 
\begin{equation}\label{eq:Xtilde}
X^*_{t+s}=x+\int_0^s a\big( X^*_{t+u},m^*(t+u)\big)\ud u+\int_0^s\sigma(X^*_{t+u})\ud W_{t+u},
\end{equation}
where $m^*$ is defined as in Theorem \ref{teo:existenceSolMFG}.
A simple corollary of the last theorem follows from the convergence of $b_n$ to $b$ (cf.\ Corollary \ref{cor:unifconv}). This corollary also motivates our numerical implementation of the solution for the integral equation of $b$ as it will be clarified later. 
\begin{corollary}\label{cor:inteq}
Letting $n\to\infty$ in \eqref{eq:integral_rep} and in \eqref{eq:integral_eq0}, those formulae hold with $(u_n,b_n,X^{[n]})$ replaced by $(u,b,X^*)$.
\end{corollary}

Before stating the uniqueness theorem, we introduce the following class of functions: for $y\in[0,1]$
\begin{eqnarray}
B_y := \{\beta\in C([0,T])\,:\,\beta(T) = \bar{x}(y)\quad\text{and}\quad \beta(t)\ge\bar{x}(y) \quad \forall t\in[0,T)\}.
\label{eq:nume:class_unique_sol}
\end{eqnarray}

\begin{theorem}\label{th:UniquenessTheoremIntegralEquation}
Suppose Assumptions \ref{itm:A1}--\ref{itm:A3} hold. Then, for all $y\in[0,1]$, the function $b_n(\cdot,y)$ is the unique solution in the class $B_y$ of the integral equation \eqref{eq:integral_eq0}. Analogously, the function $b(\cdot,y)$ is the unique solution in the class $B_y$ of \eqref{eq:integral_eq0} with $X^{[n]}$ replaced by $X^*$.
\end{theorem}
\begin{proof}
The proof follows a well-established procedure which was originally introduced in \cite{P2005MF}. Here we provide a summary of the main arguments and we prove the claim only for $b_n$, because the proof for $b$ is analogous. Let $\widehat\cH:=[0,T]\times\cH$ with $\mathcal{H}$ defined in \eqref{eq:setH}. 
Then $\widehat{\mathcal{H}}\subset\mathcal{C}^{[n]}$ with
\[
\widehat{\cH}_y=\{(t,x)\in[0,T]\times\R:\partial_y f(x,y)-rc_0<0\}\subset\cC^{[n]}_y,\quad \forall y\in[0,1].
\]
The boundary of the $y$-section $\widehat{\mathcal{H}}_y$ reads $\mathrm{cl}(\widehat{\cH}_y)\cap\mathrm{cl}(\widehat{\cH}_y^c)=[0,T]\times \bar{x}(y)$. 

Let us assume that there exists $\gamma:[0,T]\times[0,1]\rightarrow\mathbb{R}$ such that $\gamma(\cdot,y) \in B_y$ and it solves the integral equation
\begin{equation*}
\begin{aligned}
c_0 &= c_0\e^{-r(T-t)} + \E_{t,\gamma(t,y)}\Big[\int_0^{T-t}\e^{-rs}\partial_yf(X^{[n]}_{t+s},y)\,\ud s\Big]\\
&\quad+\E_{t,\gamma(t,y)}\Big[\int_0^{T-t}\e^{-rs}\left(rc_0-\partial_yf(X^{[n]}_{t+s},y)\right)\mathsf{1}_{\{X^{[n]}_{t+s}\geq \gamma(t+s,y)\}}\,\ud s\Big],
\end{aligned}
\end{equation*}
for each $(t,y)\in[0,T)\times[0,1]$. Since $n\in\N$ is fixed throughout, we simplify our notation and write $X=X^{[n]}$. We set 
\begin{equation}
\begin{aligned}
U^{\gamma}(t,x,y) &= c_0\e^{-r(T-t)} + \E_{t,x}\Big[\int_0^{T-t}\e^{-rs}\partial_yf(X_{t+s},y)\,\ud s\Big]\\
&\quad +\E_{t,x}\Big[\int_0^{T-t}\e^{-rs}\big(rc_0-\partial_yf(X_{t+s},y)\big)\mathsf{1}_{\{X_{t+s}\geq \gamma(t+s,y)\}}\,\ud s\Big],
\end{aligned}
\end{equation}
for all $(t,x,y)\in[0,T]\times\Sigma$, then trivially $U^{\gamma}(T,x,y)=c_0$ and $U^{\gamma}(t,\gamma(t,y),y)=c_0$ by definition of $\gamma$. By continuity of the stochastic flow $(t,s,x)\mapsto X^{t,x}_{t+s}(\omega)$ and of the function $\partial_y f$, combined with \eqref{eq:probab}, it is easy to show continuity of $(t,x)\mapsto U^\gamma(t,x,y)$ for each $y\in[0,1]$. 
Moreover, by the strong Markov property of $X$ it is not hard to prove that the following processes are continuous martingales for each $y\in[0,1]$:
\begin{equation}\label{eq:Umart}
\begin{aligned}
s\mapsto \e^{-rs} &U^{\gamma}(t+s,X_{t+s},y)+\int_0^{s}\e^{-ru}\partial_yf(X_{t+u},y)\,\ud u\\
&\quad+\int_0^{s}\e^{-ru}\big(rc_0-\partial_yf(X_{t+u},y)\big)\mathsf{1}_{\{X_{t+u}\geq \gamma(t+u,y)\}}\,\ud u
\end{aligned}
\end{equation}
and
\begin{equation}\label{eq:umart}
\begin{aligned}
s\mapsto \e^{-rs} &u_n(t+s,X_{t+s},y)+\int_0^{s}\e^{-ru}\partial_yf(X_{t+u},y)\,\ud u\\
&\quad+\int_0^{s}\e^{-ru}\big(rc_0-\partial_yf(X_{t+u},y)\big)\mathsf{1}_{\{X_{t+u}\geq b_n(t+u,y)\}}\,\ud u.
\end{aligned}
\end{equation}
Next, we follow the four steps in the proof of \cite{P2005MF}.
\medskip

{\em Step 1:} Here we show that $U^{\gamma}(t,x,y)=c_0$ for all $x\geq\gamma(t,y)$. The result is trivial when $x=\gamma(t,y)$ or when $t=T$. 
Let us now consider $t<T$ and $x>\gamma(t,y)$. Define
\begin{equation}
\tau_{\gamma}\coloneqq\inf\{s\geq 0\,:\,X_{t+s}\leq \gamma(t+s,y)\}\wedge(T-t).
\nonumber
\end{equation}
Then by the martingale property \eqref{eq:Umart}
\begin{equation}\label{eq:Ug0}
U^{\gamma}(t,x,y)=\E_{t,x}\Big[\e^{-r\tau_{\gamma}}U^{\gamma}(t+\tau_{\gamma},X_{t+\tau_{\gamma}},y)\Big]+\E_{t,x}\Big[ c_0\big(1-\e^{-r\tau_{\gamma}}\big)\Big].
\end{equation}
By continuity of trajectories and of the function $U^\gamma(\cdot,y)$ we have $U^{\gamma}(t+\tau_{\gamma},X_{t+\tau_{\gamma}},y)=c_0$. Plugging the latter into \eqref{eq:Ug0} we conclude this step.
\medskip

{\em Step 2:} Here we show that $U^{\gamma}(t,x,y)\geq u_n(t,x,y)$ for all $(t,x,y)\in[0,T]\times\Sigma$. The inequality is trivial for $x\geq\gamma(t,y)$ and for $t=T$, thanks to the previous step. 
Let us consider $t<T$ and $x<\gamma(t,y)$. Define
\begin{equation}
\zeta^{\gamma}\coloneqq\inf\{s\geq 0\,:\,X_{t+s}\geq \gamma(t+s,y)\}\wedge(T-t).
\end{equation}
By the martingale property \eqref{eq:Umart}
\begin{equation*}
\begin{aligned}
U^{\gamma}(t,x,y)&= \E_{t,x}\Big[\e^{-r\zeta^{\gamma}}U^{\gamma}(t+\zeta^{\gamma},X_{t+\zeta^{\gamma}},y)
+\int_0^{\zeta^{\gamma}}\e^{-ru}\partial_yf(X_{t+u},y)\,\ud u\Big]\\
&=\E_{t,x}\Big[\e^{-r\zeta^{\gamma}}c_0+\int_0^{\zeta^{\gamma}}\e^{-ru}\partial_yf(X_{t+u},y)\,\ud u\Big]\\
&\geq u_n(t,x,y),
\end{aligned}
\end{equation*}
where in the second equality we used   $U^{\gamma}(t+\zeta^{\gamma},X_{t+\zeta^{\gamma}},y)=c_0$ and the last inequality is a consequence of the definition of $u_n$.
\medskip

{\em Step 3:} Here we show that $\gamma(t,y)\leq b_n(t,y)$ for all $(t,y)\in[0,T]\times[0,1]$. By definition of the class $B_y$ it is clear that $\gamma(T,y)= b_n(T,y)=\bar x(y)$. Arguing by contradiction let us assume that there exists $t\in[0,T)$ such that $\gamma(t,y)>b_n(t,y)$ and fix $x\geq \gamma(t,y)$. Define
\begin{equation}
\tau_{b,n}\coloneqq\inf\{s\geq 0\,:\,X_{t+s}\leq b_n(t+s,y)\}\wedge(T-t).
\nonumber
\end{equation}
By the martingale properties \eqref{eq:Umart} and \eqref{eq:umart} and using that $(U^\gamma-u_n)(t,x,y)=0$ for $x\ge \gamma(t,y)>b_n(t,y)$, we have
\begin{equation*}
\begin{aligned}
0&=(U^\gamma-u_n)(t,x,y)\\
&=\E_{t,x}\Big[\e^{-r\tau_{b,n}} \big(U^{\gamma}-u_n\big)(t+\tau_{b,n},X_{t+\tau_{b,n}},y)\!-\!\!\int_0^{\tau_{b,n}}\!\!\!\!\e^{-ru}\big(rc_0-\partial_yf(X_{t+u},y)\big)\mathsf{1}_{\{X_{t+u} < \gamma(t+u,y)\}}\,\ud u\Big]\\
&\ge -\E_{t,x}\Big[\int_0^{\tau_{b,n}}\!\!\!\!\e^{-ru}\big(rc_0-\partial_yf(X_{t+u},y)\big)\mathsf{1}_{\{X_{t+u} < \gamma(t+u,y)\}}\,\ud u\Big],
\end{aligned}
\end{equation*}
where the inequality holds because $U^{\gamma} \geq u_n$, by Step 2. Since $x \geq \gamma(t,y) > b_n(t,y)$ and $s\mapsto X_{t+s}-b_n(t+s,y)$ is continuous, $\P_{t,x}(\tau_{b,n}>0)=1$. Moreover, by continuity of $b_n(\cdot,y)$ and $\gamma(\cdot,y)$ and the assumption $\gamma(t,y)>b_n(t,y)$, there exists $\delta,\varepsilon>0$ such that $\gamma(t+s,y)-b_n(t+s,y)\geq \varepsilon$, for all $s \in [0,\delta]$. Thanks to the non-degeneracy of the diffusion coefficient, with strictly positive probability the process enters the open strip $\{(s,z)\,:\,b_n(t+s,y)< z < \gamma(t+s,y)\}$ before 
$\tau_{b,n}$ and it spends a nontrivial time interval in there. Hence, there must be $0<s_0<s_1$ such that 
\[
\P_{t,x}\Big(\big\{b_n(t+s)<X_{t+s}<\gamma(t+s,y),\forall s\in[s_0,s_1]\big\}\cap\big\{s_1<\tau_{b,n}\big\}\Big)>0.
\]
Because $rc_0-\partial_yf(x,y)<0$ when $x> b_n(t,y)$, we reach a contradiction: 
\[
0\le\E_{t,x}\Big[\int_0^{\tau_{b,n}}\!\!\!\!\e^{-ru}\big(rc_0-\partial_yf(X_{t+u},y)\big)\mathsf{1}_{\{X_{t+u} < \gamma(t+u,y)\}}\,\ud u\Big]<0,
\]
and it must be $\gamma(t,y)\le b_n(t,y)$.
\medskip

{\em Step 4:} Here we show that $\gamma(t,y)= b_n(t,y)$. We know from the previous step that $\gamma\le b_n$. Arguing by contradiction let us assume that there exists $t\in[0,T)$ such that $\gamma(t,y)<b_n(t,y)$ and let $ \gamma(t,y)<x<b_n(t,y)$. Define
\begin{equation}
\zeta_{b,n}\coloneqq\inf\{s\geq 0\,:\,X_{t+s}\geq b_n(t+s,y)\}\wedge(T-t).
\end{equation} 
By the martingale properties \eqref{eq:Umart} and \eqref{eq:umart} and using that $U^\gamma(t,x,y)=c_0>u_n(t,x,y)$ for $\gamma(t,y)<x<b_n(t,y)$, we have
\begin{equation*}
\begin{aligned}
0&<(U^\gamma-u_n)(t,x,y)\\
&=\E_{t,x}\Big[\e^{-r\zeta_{b,n}} \big(U^{\gamma}-u_n\big)(t+\zeta_{b,n},X_{t+\zeta_{b,n}},y)\!+\!\!\int_0^{\zeta_{b,n}}\!\!\!\!\e^{-ru}\big(rc_0-\partial_yf(X_{t+u},y)\big)\mathsf{1}_{\{X_{t+u} \ge \gamma(t+u,y)\}}\,\ud u\Big]\\
&= \E_{t,x}\Big[\int_0^{\zeta_{b,n}}\!\!\!\!\e^{-ru}\big(rc_0-\partial_yf(X_{t+u},y)\big)\mathsf{1}_{\{X_{t+u} \ge \gamma(t+u,y)\}}\,\ud u\Big],
\end{aligned}
\end{equation*}
where the final equality holds because $\big(U^{\gamma}-u_n\big)(t+\zeta_{b,n},X_{t+\zeta_{b,n}},y)=0$ by continuity of the trajectories and of the functions $U^\gamma$, $u_n$ and $b_n(\cdot,y)$. However, $rc_0-\partial_y f(x,y)\le 0$ when $x\ge \gamma(t,y)$, because $\gamma\in B_y$, hence we reach a contradiction. This concludes the proof.
\end{proof}
\begin{remark}
If the process $X^{[n]}$ with dynamics given by \eqref{eq:dynamicsn} admits a transition density denoted $p_n(t, x, t+s, z)\doteq \partial_{z}\P(X_{t+s}^{[n]; t,x} \leq z)$, the integral equation takes the more explicit form 
\begin{equation}
\label{eq:integral_eq}
\begin{cases}
c_0(1-\e^{-r (T-t)}) = \int_0^{T-t}\left(\int_{-\infty}^{\infty} \e^{-r  s} \partial_y f(z,y) p_n(t, b_n(t,y), t+s, z)\,\,\ud z\right)\,\ud s \\
\qquad \qquad \qquad \qquad + \int_{0}^{T-t}\left(\int_{b_n(t+s,y)}^{\infty} \e^{-r s}\left(r c_0 - \partial_y f(z,y)\right)p_n(t, b_n(t,y), t+s, z)\,\ud z\right) \,\ud s,\\
b_n(T-,y) = \bar{x}(y),
\end{cases}
\end{equation}
for all $(t,y)\in [0,T)\times[0,1]$. Analogously we can argue for $X^*$ with transition density $p(t,x,t+s,z)$. 

This allows numerical solution of the equation via Picard iterations.
\end{remark}

\section{Numerical Experiments}\label{sec:numericalexperiments}

In this section, we numerically solve, in a specific example, the \textrm{MFG} of finite-fuel capacity expansion. First we adopt a particular form for the functions $a(x,m)$, $\sigma(x)$ and $f(x,y)$ and simplify the expressions in \eqref{eq:integral_eq}. Then, we further specialize and make an explicit choice of $f$ to present some numerical results. 

We consider the case in which $f(x, y):=\e^{x}\,g(y)$ with $g \in C^2((0,1])\cap C([0,1])$, $g\ge 0$, strictly concave and strictly increasing. For a generic function $m(t)$ the dynamics of $X$ in \eqref{eq:XYdynamicsnew} reads \begin{equation*}
X_{t+s}^{t,x}  = x + \int_{0}^{s} m(t+u)\,\ud u + \sigma (W_{t+s}-W_{t}).
\end{equation*}
When $m=m^*$ we denote the optimal dynamics by $X^*$ and for $m=m^{[n-1]}$ we denote the dynamics by $X^{[n]}$.

For arbitrary $m(t)$ the law of the process is Gaussian with $X_{t+s}^{t,x} \overset{d}{\sim} \mathcal{N}(x+\int_{0}^{s}m(t+u)\,\ud u, \sigma^2 s)$, where $\mathcal{N}(\alpha,\gamma^2)$ denotes a Gaussian random variable with mean $\alpha$ and variance $\gamma^2$. From now on we work with the optimal dynamics $X^*$ and consider \eqref{eq:integral_eq} with $b(t,y)$ instead of $b_n(t,y)$ and with $p(t,x,t+s,z)$ instead of $p_n(t,x,t+s,z)$. The first step is to rewrite \eqref{eq:integral_eq} in a more explicit form using the properties of the Gaussian density. In particular, we have
\[
p(t,x,t+s,z)=\frac{1}{\sqrt{2\pi\sigma^2 s}}\exp\Big(-\frac{\big(z-x-\int_0^s m^*(t+u)\ud u\big)^2}{2\sigma^2 s}\Big).
\]
Letting $\Phi(z)$ be the cumulative standard normal, we have 
\begin{equation*}
\begin{aligned}
&\,r\,c_0\,\int_{b(t+s,y)}^{\infty} \e^{-r s}p(t, b(t,y), t+s, z) \,\ud z = r\,c_0\,\e^{-r s} \left[1 - \Phi(\beta(s))\right],
\end{aligned}
\end{equation*}
where 
\begin{equation}\label{eq::betas}
    \beta(s):=\frac{b(t+s,y)-b(t,y)-\int_{0}^{s}m^{*}(t+u)\,du}{\sigma\sqrt{s}}.
\end{equation}
The two integrals on the right-hand side of \eqref{eq:integral_eq} involving $\partial_y f(x,y)=\e^x g'(y)$ can be simplified by grouping them together and rewriting
\begin{equation*}
\begin{aligned}
&\int_{-\infty}^\infty\e^{-rs}\partial_y f(z,y)p\big(t,b(t,y),t+s,z\big)\ud z-\int_{b(t+s,y)}^\infty\e^{-rs}\partial_y f(z,y)p\big(t,b(t,y),t+s,z\big)\ud z\\
&=\int^{b(t+s,y)}_{-\infty}\e^{-rs}\partial_y f(z,y)p\big(t,b(t,y),t+s,z\big)\ud z\\
&=\e^{-rs}g'(y)\int^{b(t+s,y)}_{-\infty}\e^{z}p\big(t,b(t,y),t+s,z\big)\ud z.
\end{aligned}
\end{equation*}
For the remaining integral, standard algebraic steps lead to
\begin{equation*}
\begin{aligned}
&\int^{b(t+s,y)}_{-\infty} \e^z p(t, b(t,y), t+s, z) \,\ud z \\
&=\int^{b(t+s,y)}_{-\infty} \e^{z}\frac{1}{\sqrt{2 \pi \sigma^2 s}} \e^{-\frac{1}{2 \sigma^2 s}\big(z-(b(t,y)+\int_{0}^{s}m^{*}(t+u)\,\ud u)\big)^2}\,\ud z\\
&=\e^{b(t,y)} \exp\left(\frac{\sigma^2 s}{2} + \int_{0}^{s}m^{*}(t+u)\,\ud u\right)\int^{b(t+s,y)}_{-\infty}\frac{1}{\sqrt{2 \pi \sigma^2 s}}\e^{-\frac{1}{2 \sigma^2 s}\big(z-(b(t,y)+ \sigma^2 s + \int_{0}^{s}m^{*}(t+u)\,\ud u)\big)^2}\, \ud z\\
&=\e^{b(t,y)} \exp\left(\int_{0}^{s}m^{*}(t+u)\, \ud u + \frac{\sigma^2 s}{2} - r s\right)\Phi(\beta(s)-\sigma\sqrt{s}), 
\end{aligned}
\end{equation*}
with $\beta(s)$ defined as in \eqref{eq::betas}. 
Therefore, the integral equation in \eqref{eq:integral_eq} can be written as
\begin{equation}
\begin{aligned}
c_0(1-\e^{-r(T-t)}) &= r c_0 \int_{0}^{T-t} \e^{-r s} \left[1 - \Phi(\beta(s))\right] \,\ud s \\
&\quad+ g'(y) \e^{b(t,y)} \int_{0}^{T-t}\exp\left(\int_{0}^{s}m^{*}(t+u) \ud u + \frac{\sigma^2 s}{2} - r s\right)\Phi(\beta(s)-\sigma\sqrt{s})\ud s\\
& \eqqcolon  rc_0 I^{(1)}(t,b(\cdot,y);T,r,\sigma)+g'(y)\e^{b(t,y)}I^{(2)}(t,b(\cdot,y);T,r,\sigma).
\end{aligned}
\end{equation}
Rearranging terms and taking logarithms we obtain
\begin{equation}\label{eq:inteqb}
\begin{aligned}
b(t,y) =\log{c_0} &+\log{\left((1-\e^{-r(T-t)})- rI^{(1)}(t,b(\cdot,y);T,r,\sigma) \right)}\\
		 &-\log{g'(y)}-\log{I^{(2)}(t,b(\cdot,y);T,r,\sigma)},
\end{aligned}
\end{equation}
with $(1-\e^{-r(T-t)})- rI^{(1)}(t,b(\cdot,y);T,r,\sigma) >  0$. The same equation holds for the boundary $b_n(t,y)$ in the form
\begin{equation}\label{eq:inteqbn}
\begin{aligned}
b_n(t,y) =\log{c_0} &+\log{\left((1-\e^{-r(T-t)})- rI^{(1)}_n(t,b_n(\cdot,y);T,r,\sigma) \right)}\\
		 &-\log{g'(y)}-\log{I^{(2)}_n(t,b_n(\cdot,y);T,r,\sigma)},
\end{aligned}
\end{equation}
where the functions $I^{(1)}_n(\cdot)$ and $I^{(2)}_n(\cdot)$ are defined analogously to $I^{(1)}$ and $I^{(2)}$ but with $m^*$ replaced by $m^{[n-1]}(t)$ and $\beta$ replaced by $\beta_n(s)=(b_n(t+s,y)-b_n(t,y)-\int_0^sm^{[n-1]}(t+u)\ud u)/(\sigma \sqrt{s})$, $s>0$, and $\beta_n(0):=0$. Letting $n\to\infty$ it is easy to verify that \eqref{eq:inteqbn} converges to \eqref{eq:inteqb} thanks to Corollary \ref{cor:unifconv}. 

The difficulty with a direct numerical resolution of \eqref{eq:inteqb} lies in the presence of the mean-field interaction term $m^*(t)$, which depends in a complex way on the unknown boundary $b$ via the trajectory of optimally controlled dynamics. Instead, a numerical solution of \eqref{eq:inteqbn} for $b_n$ is an easier task, because the function $m^{[n-1]}(t)$ depends only on $b_{n-1}$. Then, our strategy is to solve \eqref{eq:inteqbn} numerically  for each $n=1,2,\ldots $ using a Picard iteration scheme, and then choose large $n$ to approximate the optimal boundary $b$. More precisely, we use two nested iterations. The index $n$ denotes what we call a \emph{game iteration} (outer learning loop) producing $(b_n, m^{[n]})$ from $m^{[n-1]}$. For each fixed $n$, we solve the integral equation for $b_n$ by a Picard iteration scheme indexed by $k$ (i.e., for each Picard iteration we obtain $b^{(k)}_n$). The theoretical description of the algorithm is provided in the next sections and the pseudo-code is given in \textbf{Algorithm} \ref{alg:sec5}. 

\subsection{Picard scheme for the integral equation}

We truncate the state space and consider a bounded domain $[0,T]\times\Sigma_p$ where 
$\Sigma_p=[x_m,x_M] \times [y_0,1]$ for suitably chosen $x_m<x_M$ and some $y_0>0$; we take $y_0>0$ because $g'(y)$ may diverge as $y \downarrow 0$. Then, we discretise time and space according to equally spaced partitions: given $\ell_1,\ell_2,\ell_3\in\N$ we let
$\Pi_{t}=\{0=t_0<t_1<\ldots<t_{\ell_1}=T\}$, $\Pi_{y}=\{0<y_0<y_1<\ldots<y_{\ell_2}=1\}$ and $\Pi_{x}=\{x_m=x_0<x_1<\ldots<x_{\ell_3}=x_M\}$ with mesh $\Delta_t=T/\ell_1$, $\Delta_y=(1-y_0)/\ell_2$ and $\Delta_x=(x_M-x_m)/\ell_3$. 

At the $0$-th step of the iterative scheme (i.e. for $b_0(t,y)$) we have  $m^{[-1]}\equiv 1$ and the quantities $I^{(1)}_0$ and $I^{(2)}_0$ in \eqref{eq:inteqbn} above can be considerably simplified.  In order to numerically solve \eqref{eq:inteqbn} for $b_0(t,y)$, we initialize the algorithm by setting $b^{(0)}_{0}(t_i, y_j)=\bar x(y_j)=\log(r c_0)-\log(g'(y_j))$, for all $i=0,1,\ldots, \ell_1$, $j = 0,1, \ldots, \ell_2$. Notice that the choice is motivated by the terminal condition $b_n(T-,y)=\bar x(y)$ in \eqref{eq:inteqbn}. Other choices are of course possible and they will in general affect the overall speed of convergence of the Picard scheme at the subsequent steps.

We denote by $b^{(k)}_{0}(t_i,y_j)$ the value of the boundary obtained from the $k$-th iteration of the Picard scheme. Then $b^{(k+1)}_{0}(t_i,y_j)$ is computed as 
\begin{equation}\label{eq:numerics0}
\begin{split}
b_{0}^{(k+1)}(t_i,y_j) &=\log{c_0} +\log{\left((1-\e^{-r(T-t_i)})- rI^{(1)}_0(t_i,b_{0}^{(k)}(\cdot,y_j);T,r,\sigma) \right)}\\
		 &\quad-\log{g'(y_j)}-\log{I^{(2)}_0(t_i,b_{0}^{(k)}(\cdot,y_j)};T,r,\sigma),
\end{split}
\end{equation}
for $i=0,1,\ldots, \ell_1$ and $j=0,1,\ldots,\ell_2$. All the integrals with respect to the time variable are computed by a standard quadrature method using the same grid $\Pi_t$ (so we only require values of the boundary $b_0^{(k)}(t_i+q\Delta t, y_j)$, with $q=0, \ldots \ell_1-i$).
The iteration terminates when we reach a sufficiently large $k_0$ so that
\begin{equation}\label{eq:tolerance}
\|b^{(k_0)}_0-b^{(k_0-1)}_0\|_2:=\Big(\sum_{i,j}\big|b^{(k_0)}_0(t_i,y_j)-b^{(k_0-1)}_0(t_i,y_j)\big|^2\Big)^\frac12<\eta,
\end{equation}
where $\eta>0$ is a small number that represents our tolerance on the approximation. Then, we take $b_0^{k_0}$ as our proxy for $b_0$ and we denote it by $b_0$ hereafter. 

In order to proceed to the numerical approximation for $b_1$ we first need to calculate  $\xi_t^{[0]*}=\sup_{0 \leq s \leq t}(c_0(s, X^{[0];x}_{s})-y)^{+}$, where $c_0(t,x)$ is the generalized, right-continuous inverse of $b_0(t,\cdot)$ (with respect to the $y$ variable, i.e., $c_0(t,x)=\inf\{y:b_0(t,y)>x\}$). The latter is calculated by numerical inversion using the Python module \texttt{pynverse}, which is specialized in computing the numerical inverse of any invertible continuous function. 
Then $Y_{t}^{[0]*}=y+\xi^{[0]*}$ and $m^{[0]}(t)$ is obtained as
\[
m^{[0]}(t)=\int_\Sigma\E_{x,y}\big[Y^{[0]*}_t\big]\nu(\ud x,\ud y)\approx  \int_{\Sigma_p}\E_{x,y}\big[Y^{[0]*}_t\big]\nu(\ud x,\ud y),
\] 
by Monte Carlo simulation. It is worth noticing that when computing the generalized inverse $c_0(t,x)$, the latter is set equal to $1$ if $x > \max_{y \in \Pi_y} b_{0}(t,y)$ and to $0$ if $x < \min_{y \in \Pi_y} b_0(t,y)$. In the numerical implementation $\nu$ is taken as the uniform discrete distribution on $\Pi_x\times\Pi_y$, because the product structure of such law helps reducing the Monte Carlo variance.

The procedure described above yields $m^{[0]}$, which we can now use to define the functions $I^{(1)}_1$ and $I^{(2)}_1$ that appear in the integral equation for $b_1$. At this point we can solve the analogue of \eqref{eq:numerics0} but with $b^{(k)}_0$, $b^{(k+1)}_0$, $I^{(1)}_0$ and $I^{(2)}_0$ therein replaced by $b^{(k)}_1$, $b^{(k+1)}_1$, $I^{(1)}_1$ and $I^{(2)}_1$. The Picard iteration terminates when $k_1$ is large enough to guarantee $\|b_1^{(k_1)}-b_1^{(k_1-1)}\|_2<\eta$ as in \eqref{eq:tolerance}. Iterating the whole procedure we obtain a sequence $(b^{(k_n)}_n)_{n\in\N}$ which we truncate when $n$ is large enough to guarantee $\|b^{(k_n)}_n-b^{(k_{n-1})}_{n-1}\|_2<\eta$.

We propose two diagnostics for the final equilibrium candidate. The numerically optimal surface is denoted by $b^{\star}$ and the corresponding driver mean is denoted by $m^{\dag}$. Then we consider the residual 
\begin{equation}\label{eq::residual}
    \mathcal{R}(t_i, y_j):=|\Phi(b^{\star};m^{\dag})(t_i,y_j)-b^{\star}(t_i, y_j)|, 
\end{equation}
for $i=0,1,\ldots,\ell_1$ and $j=0,1,\ldots,\ell_2$, where $\Phi$ denotes the right-hand side of \eqref{eq:inteqbn} evaluated using $b^\star$ and $m^\dag$.
We introduce two norms: 
\begin{equation}
\|\mathcal{R}\|_{\infty}:=\max_{i,j}\mathcal{R}(t_i, y_j),\quad\quad \|\mathcal{R}\|_{2}:=\left(\sum_{i,j}\mathcal{R}(t_i, y_j)^2\right)^{1/2}. 
\end{equation}
Along simulated optimal paths we compute 
\begin{equation}\label{eq::diagnostic_two}
    G_{t_i}:=Y_{t_i}^{\star}-c^\star({t_i},X_{t_i}^{\star}),\quad i=0,1,\ldots \ell_1,
\end{equation}
with $c^\star(t,\cdot)$ the generalized inverse of $b^\star(t,\cdot)$. Increments of the control on the discrete time-grid are denoted $\Delta \xi_{t_i}^{\star}$. The Skorokhod conditions amounts to $G_{t_i}\ge 0$ for $i=0,1,\ldots \ell_1$ and
\begin{equation}\label{eq::diagnostic_two_one}
    \Delta \xi_{t_i}^{\star}>0 \Rightarrow Y_{t_i}^{\star}=c(t_i,X_{t_i}^{\star}).
\end{equation}
Numerically, this is verified by checking that the maximum value of $|G_t|$ on the active set $\{\Delta \xi_t^{\star} > 0\}$ is at the machine precision level.

\subsection{A concrete example} In our numerical exercise, we choose $g(y)=\sqrt{y}$ and fix $c_0=0.5$, $r=0.01$, $T=1$, and $\sigma=1$. To discretise time and space, we set $\ell_1=75$, $\ell_2=50$, $x_m=-5$ and $x_M=0.5$ with $\ell_3=25$. The values for $x_m$ and $x_M$ have been chosen so as to guarantee that the simulated paths of the initial process $X^{[0]}$ stay inside the interval $[x_m,x_M]$ with high probability.

In order to avoid issues with the derivative of $g(y)$ at the origin we fix $y_0=10^{-3}$, and $\eta=10^{-3}$. The maximum number of Picard iterations has been fixed a priori to $5$ and we take $b_0^{(5)}(t,y)$ as optimal; we verify a posteriori that for all the considered $n$ it indeed holds that $\|b_{n}^{(5)}-b_{n}^{(4)}\|_2 < 10^{-3}$. We now present our results. 

Figure \ref{fig:boundary_0} (\textit{left panel}) shows the boundary $b^{(0)}_{0}(t,y)$ (in \textit{red}) and the boundary $b_0^{(5)}(t,y)$ (in \textit{blue}), whereas Figure \ref{fig:boundary_0} (\textit{right panel}) shows $c_0^{(5)}(t,x)$, obtained by numerical inversion of $b_0^{(5)}(t,y)$. Figure \ref{fig:boundary_5} (\textit{left panel}) shows the boundary $b_5^{(5)}(t,y)$ and Figure \ref{fig:boundary_5} (\textit{right panel}) shows $c_5^{(5)}(t,x)$, obtained by numerical inversion of $b_5^{(5)}(t,y)$. Although we stop at game iteration $n=5$ and, for each game iteration, we stop the Picard scheme at iteration $k=5$, we observe convergence of the optimal stopping surfaces already after the first very few iterations in both cases. 

Figure \ref{fig:convergence_0} displays the numerical convergence of the Picard scheme for the boundaries $b_0^{(k)}(t,y)$, $k\in \{1,\ldots,5\}$, where the numerical error at step $k$ is defined by the 2-norm $\lVert b_0^{(k)} - b_0^{(k-1)}\rVert_2$. Similarly, Figure \ref{fig:convergence_5} displays the numerical convergence of the Picard scheme for the boundaries $b_5^{(k)}(t,y)$,  $k\in \{1,\ldots,5\}$. Both plots show a rapid decrease of the numerical error within the first $3$ iterations. 
Figure \ref{fig:convergence_6} shows for each $n\in\{1,\ldots 5\}$ the convergence of $\|b^{(5)}_n-b^{(k)}_n\|_2$ for $k\in\{1,\ldots 5\}$.

Figure \ref{fig:game_error} displays the distance between the optimal boundaries of subsequent iterations for the construction of the solution to our MFG. That is, we plot $\lVert b_n^{(5)} - b_{n-1}^{(5)}\rVert_2$ for $n\in\{1,\ldots,5\}$. That plot should be considered along with Figure \ref{fig:m_convergence_t}, where we plot $t\mapsto m^{[n]}(t)$ for $n\in \{0,\ldots,5\}$.

For a representative initial condition $(X_0,Y_{0-})=(-5,0.2)$ and a single controlled trajectory on a refined time grid with $500$ time steps, Figure \ref{fig:optimal_control} shows the following quantities: (1) the optimal sample path $X_t^{\star}$; (2) the optimal capacity $Y_t^{\star}$ together with the moving base-capacity $c^\star(t,X_t^{\star})$; (3) the cumulative control $\xi_t^{\star}$ and the distance between the optimal capacity and the moving base-capacity $Y_t^{\star}-c^\star(t,X_t^{\star})$. This figure makes the Skorokhod reflection mechanism explicit. Indeed, the control acts only when the target attempts to exceed the current fuel level, and it acts by the minimal amount required to restore the constraint.

Finally, Figure \ref{fig:diagnostic} displays the results of the diagnostic tests introduced above. The \textit{left panel} shows $ \mathcal R(t_i,y_j)=|\Phi\bigl(b_5^{(5)};m^{[4]}\bigr)(t_i,y_j)-b_5^{(5)}(t_i,y_j)|$ on the $(t,y)$ grid; see \eqref{eq::residual}. The \textit{right panel} shows the empirical values of $|G_t|$ in \eqref{eq::diagnostic_two} at the time points where the control increment is positive, together with a tolerance line (this test uses an additional Monte Carlo batch of $96$ reflected paths on a refined grid with $700$ time steps). 
In our numerical simulations, we also obtain extremely good results from tests performed according to the metrics introduced above. In particular, we observe
$\|\mathcal{R}\|_{\infty}=2.69 \cdot 10^{-4}$, $\|\mathcal{R}\|_{2}=8.82 \cdot 10^{-5}$, $\max_{t:\Delta \xi_t>0}|G_t| = 10^{-9}$, and $\min_t G_t=10^{-12}$. 
That contributes to strong numerical evidence that the computed reflected control is indeed the optimal one associated with the final free boundary.

\begin{figure}[t!]
    \centering
    \includegraphics[height=3.2in]{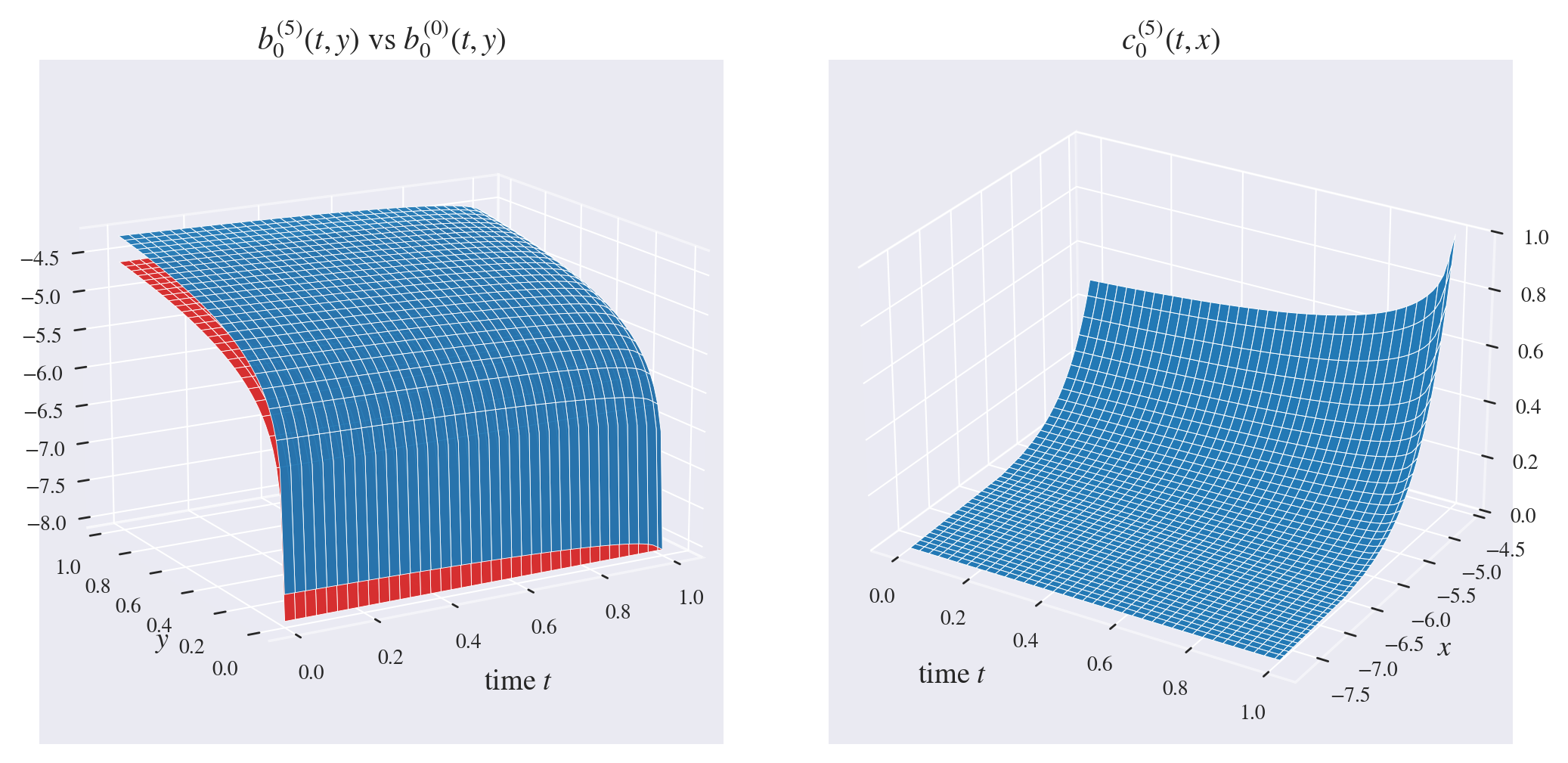}
\caption{\textit{Left panel:} The figure displays the boundary $b_0^{(0)}(t,y)$ (in \textit{red}) and the boundary $b_0^{(5)}(t,y)$ (in \textit{blue}). \textit{Right panel:} Optimal surface at game iteration $n=0$.}
\label{fig:boundary_0}
\end{figure}

\begin{figure}[t!]
    \centering
    \includegraphics[height=3.2in]{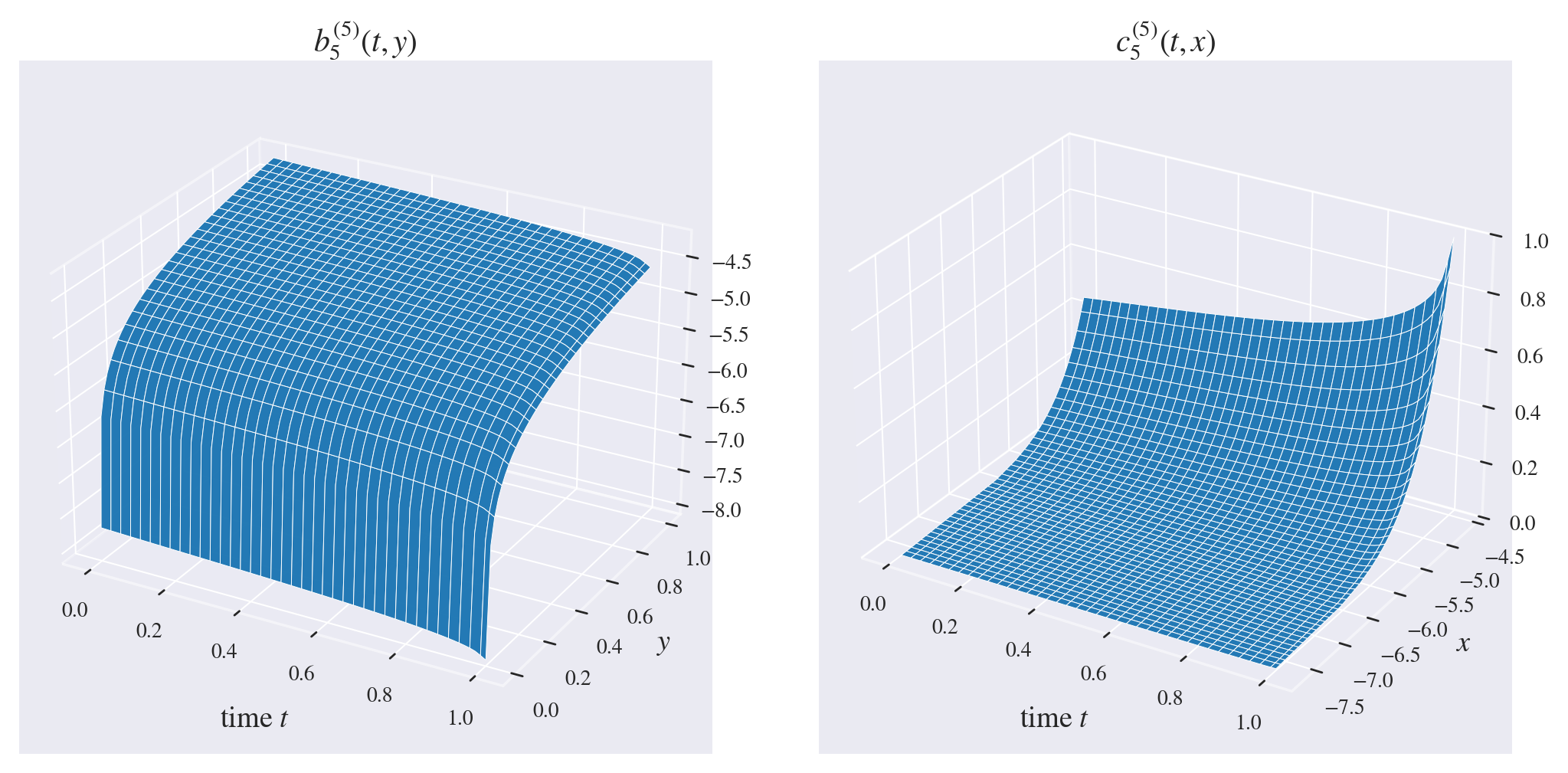}
\caption{\textit{Left panel:} The figure displays the boundary $b_5^{(5)}(t,y)$. \textit{Right panel:} The figure displays the generalized inverse $c_5^{(5)}(t,x)$.}
\label{fig:boundary_5}
\end{figure}

\begin{center}
\begin{figure}[t!]
\centering
\includegraphics[height=3.2in]{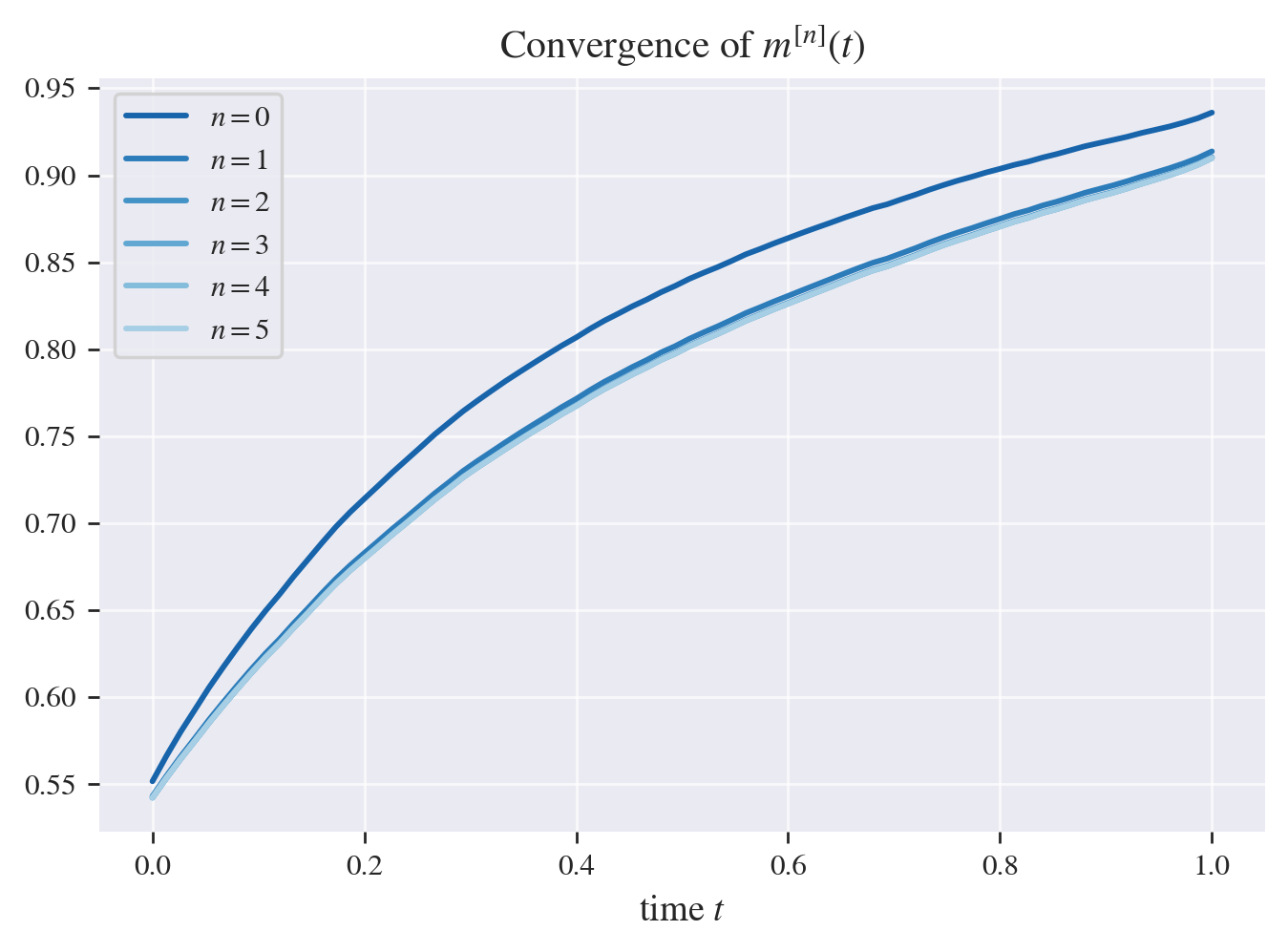}
\caption{The figure displays $m^{[n]}(t)$ for game iteration $n\in \{0,\ldots,5\}$.}
\label{fig:m_convergence_t}
\end{figure}
\end{center}

\begin{center}
\begin{figure}[t!]
\centering
\includegraphics[height=3.2in]{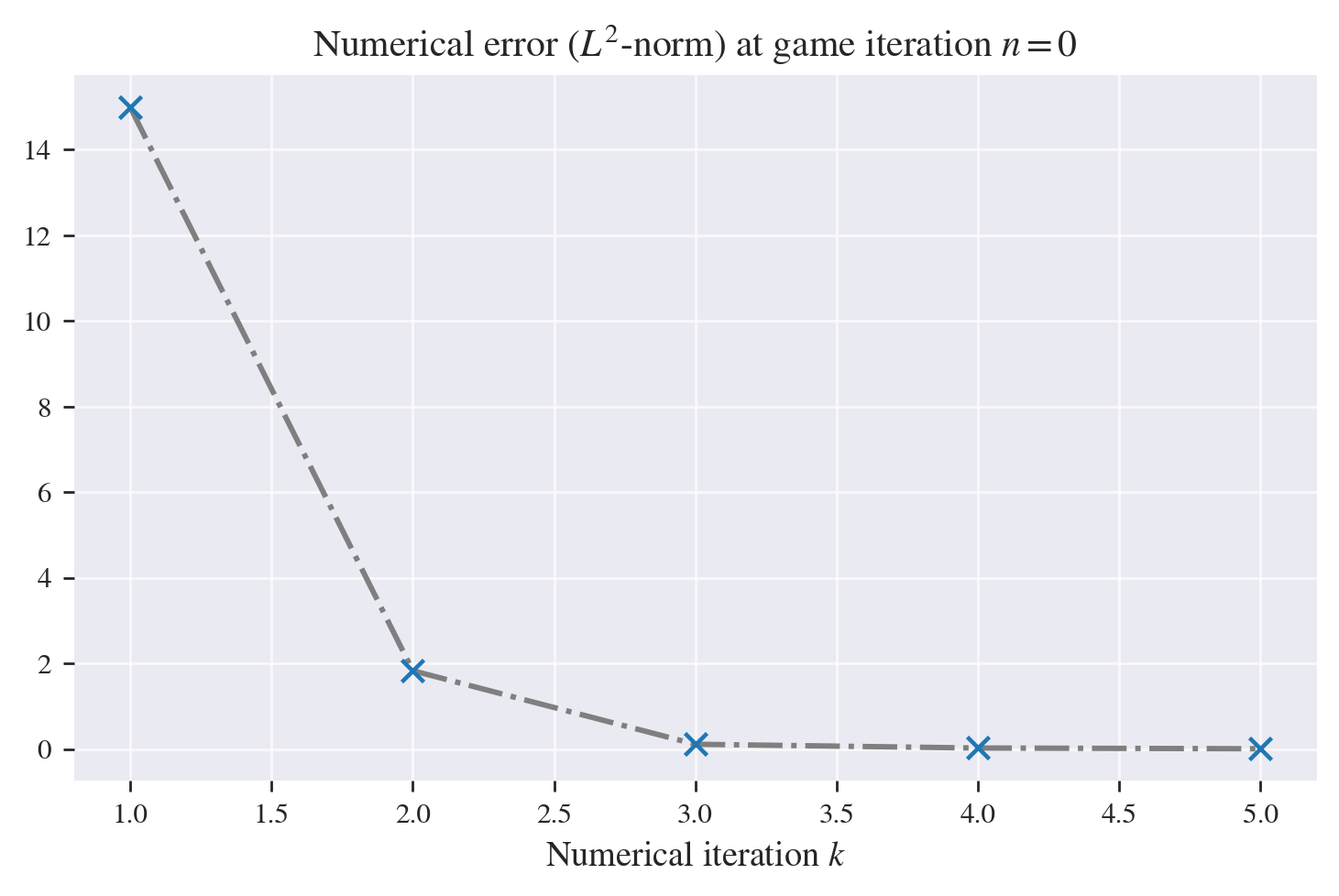}
\caption{The figure displays the numerical convergence of the boundaries $b_0^{(k)}(t,y)$ for $k\in \{1,\ldots,5\}$, where the numerical error at step $k$ is defined as the $2$-norm $\lVert b_0^{(k)} - b_0^{(k-1)}\rVert_2$.}
\label{fig:convergence_0}
\end{figure}
\end{center}

\begin{center}
\begin{figure}[t!]
\centering
\includegraphics[height=3.2in]{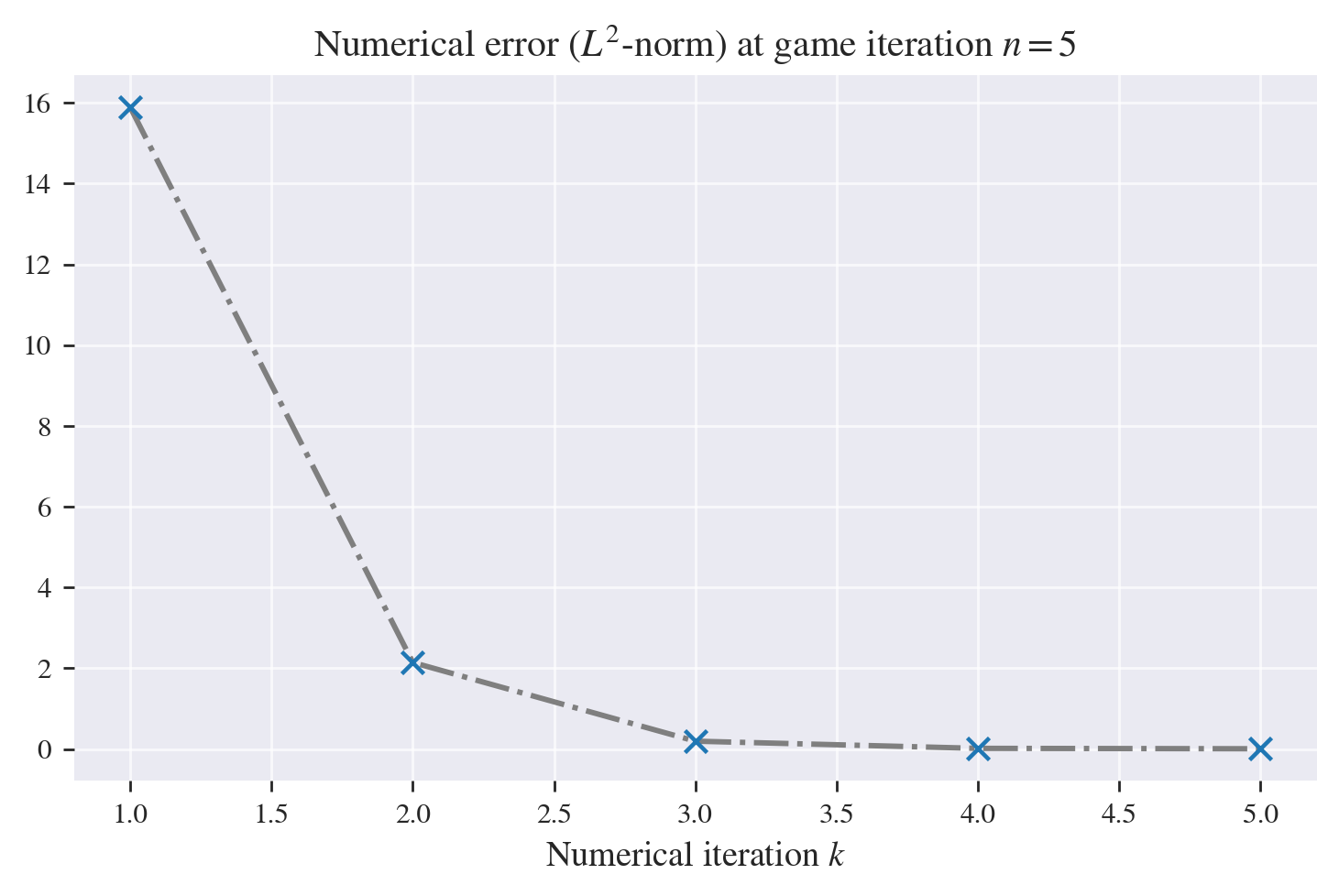}
\caption{The figure displays the numerical convergence of the boundaries $b_5^{(k)}(t,y)$ for $k\in \{1,\ldots,5\}$, where the numerical error at step $k$ is defined as the 2-norm $\lVert b_5^{(k)}- b_5^{(k-1)}\rVert_2$.}
\label{fig:convergence_5}
\end{figure}
\end{center}

\begin{center}
\begin{figure}[t!]
\centering
\includegraphics[height=3.2in]{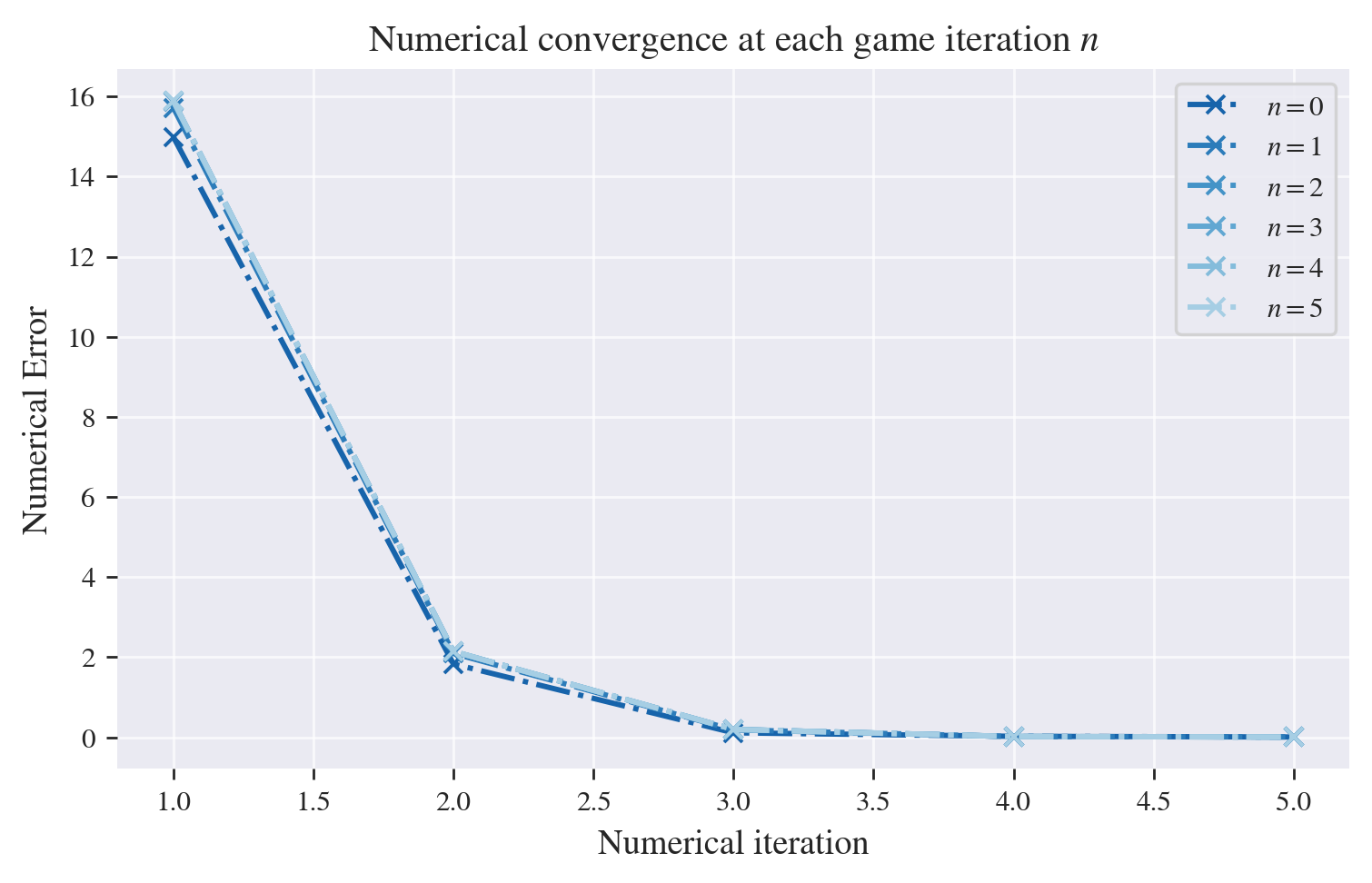}
\caption{The figure displays, for each $n\in\{1,\ldots 5\}$, the convergence of $\|b^{(5)}_n-b^{(k)}_n\|_2$ for $k\in\{1,\ldots 5\}$.}
\label{fig:convergence_6}
\end{figure}
\end{center}

\begin{center}
\begin{figure}[t!]
\centering
\includegraphics[height=3.2in]{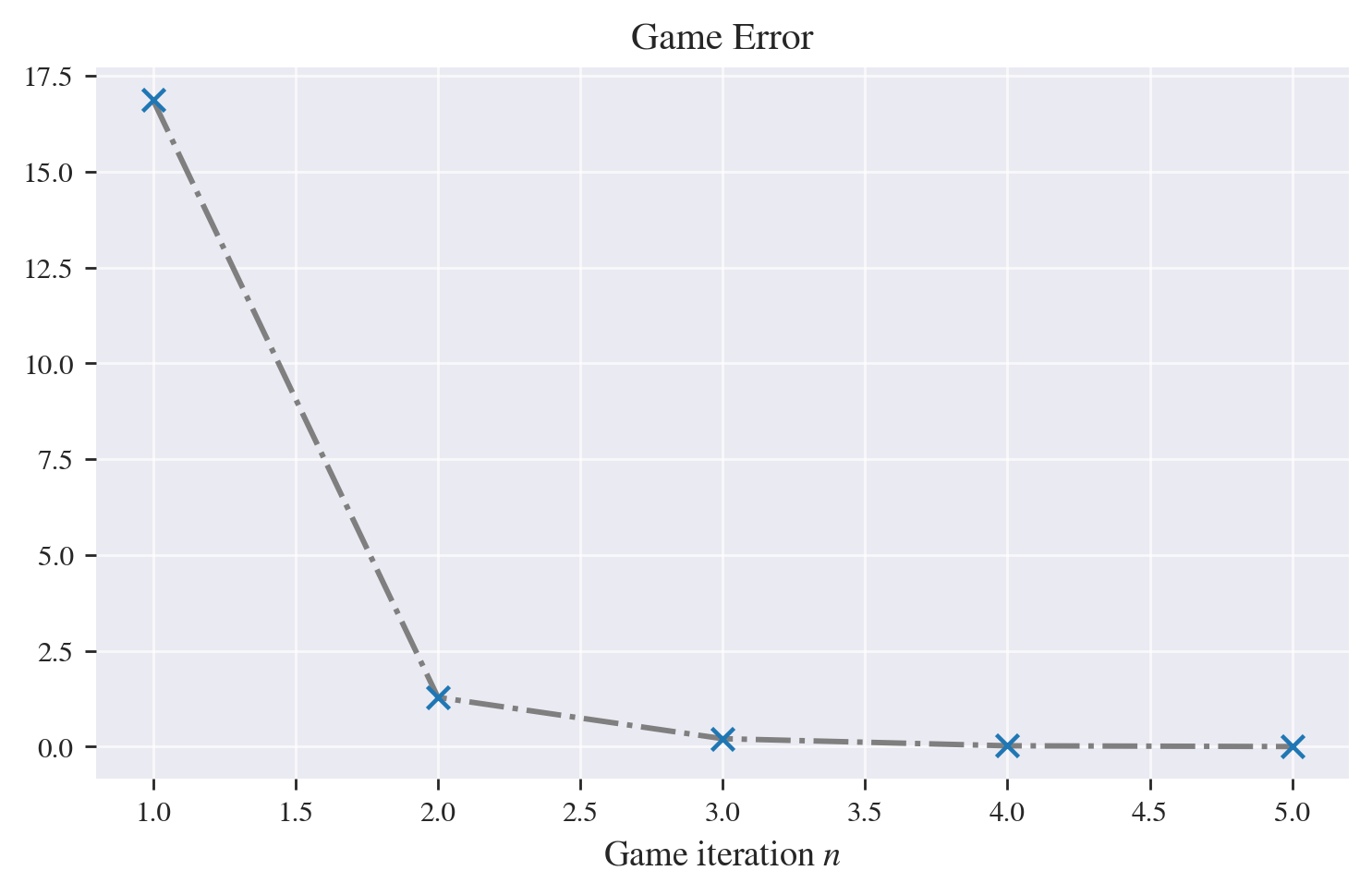}
\caption{The figure displays the distance in $2$-norm between the optimal boundaries of subsequent game iterations, i.e. $\lVert b_n^{(5)} - b_{n-1}^{(5)}\rVert_2$, $n \in \{1,\ldots,5\}$}
\label{fig:game_error}
\end{figure}
\end{center}

\begin{center}
\begin{figure}[t!]
\centering
\includegraphics[scale=0.60]{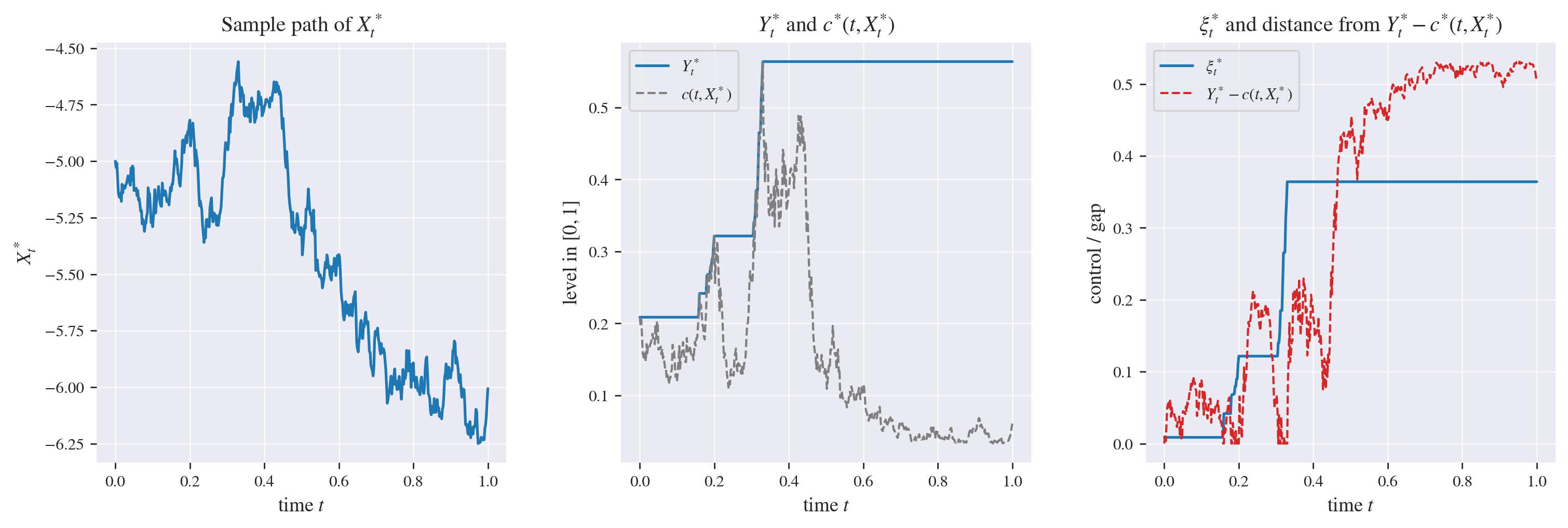}
\caption{\textit{Left panel:} The sample path $X_t^{\star}$. \textit{Center panel:} The reflected state $Y_t^{\star}$ together with the moving target $c^\star(t,X_t^{\star})$. \textit{Right panel:} The cumulative control $\xi_t^{\star}$ and the distance from the base capacity $Y_t^{\star}-c^\star(t,X_t^{\star})$.}
\label{fig:optimal_control}
\end{figure}
\end{center}

\begin{center}
\begin{figure}[t!]
\centering
\includegraphics[scale=0.75]{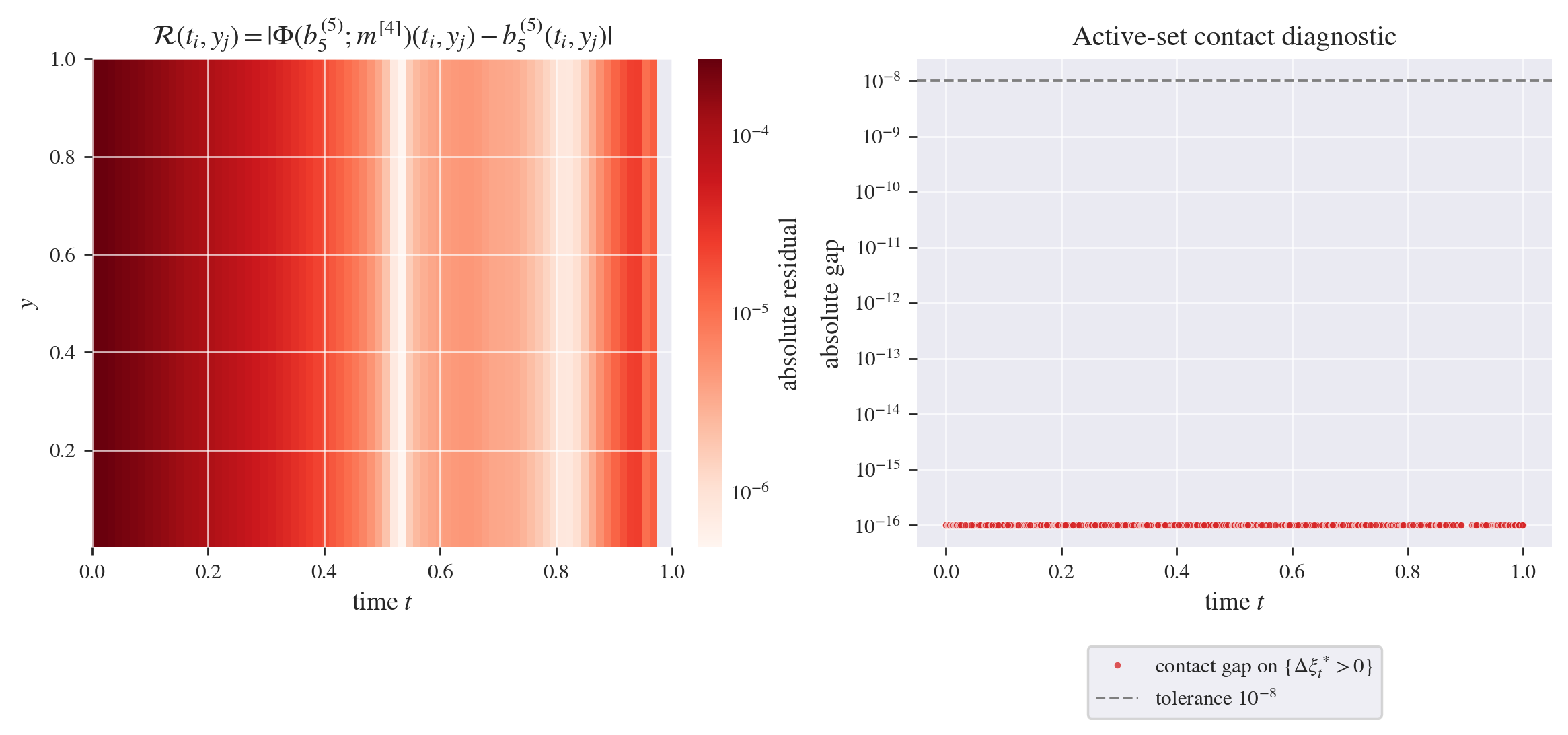}
\caption{\textit{Left panel:} The figure shows $\mathcal R(t_i,y_j)
    = |\Phi\bigl(b_5^{(5)};m^{[4]}\bigr)(t_i,y_j)-b_5^{(5)}(t_i,y_j)|$. \textit{Right panel:} Observed values of $|G_{t_i}|$ at the time points where the control increment is positive, together with a tolerance line.}
\label{fig:diagnostic}
\end{figure}
\end{center}

\clearpage

\newpage
\begin{small}
\begin{algorithm}
\caption{\small{Numerical scheme of Section 5 (game iteration and Picard solver)}}
\label{alg:sec5}
\begin{algorithmic}[1]
\REQUIRE 
Horizon $T$, parameters $(r,\sigma,c_0)$; payoff $g$ and derivative $g'$;
truncation $[x_m,x_M]$; grid sizes $(\ell_1,\ell_2,\ell_3)$;
tolerance $\eta$; max iterations $(N_{\max},K_{\max})$; Monte Carlo size $N_{\mathrm{MC}}$.
\ENSURE Approx.\ boundary $b_n(t_i,y_j)$, generalized inverse $c_n(t_i,x_k)$, mean field $m^{[n]}(t_i)$.
\STATE \textbf{Build grids on $[0,T]\times (0,1]\times[x_m,x_M]$:}
\STATE $\Delta t \gets T/\ell_1$;\quad $t_i \gets i\Delta t$ for $i=0,\dots,\ell_1$.
\STATE Choose $y_0>0$;\quad $\Delta y \gets (1-y_0)/\ell_2$;\quad $y_j \gets y_0+j\Delta y$ for $j=0,\dots,\ell_2$.
\State $\Delta x \gets (x_M-x_m)/\ell_3$;\quad $x_k \gets x_m+k\Delta x$ for $k=0,\dots,\ell_3$.
\STATE Terminal boundary (from $\partial_y f(\bar x(y),y)=rc_0$):
\STATE $\bar x(y_j)\gets \log(rc_0)-\log(g'(y_j))$.
\STATE \textbf{Initialize outer (game) loop:}
\STATE $m^{[-1]}(t_i)\gets 1$ for all $i$. ($n=0$)
\STATE $b_{-1}(t_i,y_j)\gets \bar x(y_j)$ for all $(i,j)$. ($n=0$)

\For{$n=0,1,\dots,N_{\max}$}
    \STATE \textbf{Inner loop: Picard iterations for the integral equation (54):}
    \STATE Initialize $b_n^{(0)}(t_i,y_j)\gets \bar x(y_j)$ for all $(i,j)$ ($n=0$) or with the optimal $b$ from the previous step $(n>0)$.
    \FOR{$k=0,1,\dots,K_{\max}-1$}

        \STATE Precompute cumulative integral of $m^{[n-1]}$ on $\Pi_t$:
        \STATE $M_0\gets 0$ and for $i=1,\dots,\ell_1$ set
        \STATE $M_i \gets M_{i-1} + \tfrac12\bigl(m^{[n-1]}(t_i)+m^{[n-1]}(t_{i-1})\bigr)\Delta t$\hfill $M_i\approx\int_0^{t_i} m^{[n-1]}(u)\,du$

        \FORALL{$(i,j)$}
            \STATE $I_1\gets 0$; \quad $I_2\gets 0$.
            \FOR{$q=0,1,\dots,\ell_1-i$}\hfill $s_q=q\Delta t\in[0,T-t_i]$
                \STATE $s\gets q\Delta t$.
                \STATE $\Delta M \gets M_{i+q}-M_i$\hfill $\Delta M\approx\int_0^{s} m^{[n-1]}(t_i+u)\,du$

                \IF{$s=0$}
                    \STATE $\beta \gets 0$.
                \ELSE
                    \STATE $\beta \gets
                    \dfrac{b_n^{(k)}(t_{i+q},y_j)-b_n^{(k)}(t_i,y_j)-\Delta M}{\sigma\sqrt{s}}$.
                \ENDIF

                \STATE $I_1 \gets I_1 + e^{-rs}\bigl(1-\Phi(\beta)\bigr)\,\Delta t$.
                \STATE $I_2 \gets I_2 + \exp\!\bigl(\Delta M+\tfrac12\sigma^2 s-rs\bigr)\,
                \Phi(\beta-\sigma\sqrt{s})\,\Delta t$.
            \ENDFOR

            \STATE $A \gets (1-e^{-r(T-t_i)}) - r I^{(1)}(t_i, b_n^{(k)}(\cdot,y_j),s;T, r, \sigma)$.\hfill Must be $A>0$
            \STATE Update (Picard) boundary:
            \STATE $b_n^{(k+1)}(t_i,y_j)\gets \log c_0+\log(A)-\log(g'(y_j))-\log(I^{(2)}(t_i, b_n^{(k)}(\cdot,y_j),s;T, r, \sigma)$. 
        \ENDFOR

        \STATE Compute Picard error (discrete version of (56)): $\mathrm{err}\gets \left(\sum_{i,j}\bigl|b_n^{(k+1)}(t_i,y_j)-b_n^{(k)}(t_i,y_j)\bigr|^2\right)^{1/2}$.
        \IF{$\mathrm{err}<\eta$}
            \STATE \textbf{break}
        \ENDIF    
    \ENDFOR

    \STATE Set $b_n(t_i,y_j)\gets b_n^{(k+1)}(t_i,y_j)$ for all $(i,j)$.
\algstore{myalg}
\end{algorithmic}
\end{algorithm}
\end{small}

\begin{small}
\begin{algorithm}[t]
\begin{algorithmic}[1]
\algrestore{myalg}
    \STATE \textbf{Compute generalized inverse $c_n(t,x)=\inf\{y:\,b_n(t,y)>x\}$ on $\Pi_t\times\Pi_x$ via the Python module \texttt{pynverse}}
       \STATE \textbf{Monte Carlo estimate of $m^{[n]}(t)=\mathbb E[Y_t^{[n]\ast}]$:}
    \STATE $m^{[n]}(t_i)\gets 0$ for all $i$.
    \FOR{$p=1,\dots,N_{\mathrm{MC}}$}
        \STATE Sample $(X_0,Y_{0-})$ uniformly from $\Pi_x\times\Pi_y$.
        \STATE $S\gets 0$.\hfill $S$ stores $\sup_{0\le s\le t}(c_n(s,X_s)-Y_{0-})_+$
        \FOR{$i=0,1,\dots,\ell_1$}
            \STATE $c \gets \textsc{Interp}\bigl(x\mapsto c_n(t_i,x),\,X_{t_i}\bigr)$.
            \STATE $S \gets \max\{S,\,(c-Y_{0-})_+\}$;\quad $Y_{t_i}\gets Y_{0-}+S$. 
            \STATE $m^{[n]}(t_i)\gets m^{[n]}(t_i)+Y_{t_i}$.
            \IF{$i<\ell_1$}
                \STATE Draw $Z\sim\mathcal N(0,1)$.
                \STATE $X_{t_{i+1}}\gets X_{t_i}+m^{[n-1]}(t_i)\Delta t+\sigma\sqrt{\Delta t}\,Z$.
            \ENDIF
        \ENDFOR
    \ENDFOR
    \STATE $m^{[n]}(t_i) \gets$ \text{Update the mean-field by Monte Carlo averaging.}
    \STATE \textbf{Outer stopping test (game iteration):}
    \STATE $\mathrm{err}_{\mathrm{game}}\gets \left(\sum_{i,j}\bigl|b_n(t_i,y_j)-b_{n-1}(t_i,y_j)\bigr|^2\right)^{1/2}$.
    \IF{$n\ge 1$ \textbf{and} $\mathrm{err}_{\mathrm{game}}<\eta$}
        \STATE \textbf{return} $(b_n,c_n,m^{[n]})$.
    \ENDIF

\EndFor

\STATE \textbf{return} $(b_n,c_n,m^{[n]})$.
\end{algorithmic}
\end{algorithm}
\end{small}

\newpage
\section{Declaration of generative AI and AI-assisted technologies in the manuscript preparation process.}
\noindent During the preparation of this work the authors used ChatGPT Pro to optimize some numerical routines of the \texttt{Python} code. After using this tool/service, the authors reviewed and edited the content as needed and take full responsibility for the content of the published article.

\end{document}